\documentclass{amsart}

\usepackage{amssymb}
\usepackage{hyperref}
\hypersetup{
    citecolor=blue,
    filecolor=blue,
    linkcolor=blue,
    urlcolor=blue
}

\newtheorem{thm}{Theorem}[section]
\newtheorem{prop}[thm]{Proposition}
\newtheorem{lem}[thm]{Lemma}
\newtheorem{cor}[thm]{Corollary}

\newtheorem{claim}[thm]{Claim}

\newtheorem{assumption}[thm]{Assumption}

\theoremstyle{definition}
\newtheorem{definition}[thm]{Definition}

\theoremstyle{remark}
\newtheorem{remark}[thm]{Remark}

\numberwithin{equation}{section}

\newcommand{\R}{\mathbb{R}}  
\newcommand{\N}{\mathbb{N}}  
\newcommand{\Sph}{\mathbb{S}}  
\renewcommand{\phi}{\varphi}
\newcommand{\ol}[1]{\overline{#1}} 
\newcommand{\cl}[1]{\mathcal{#1}} 
\newcommand{\C}{\mathbf{C}}  
\renewcommand{\subset}{\subseteq}		

\newcommand{\mres}{\mathbin{\vrule height 1.6ex depth 0pt width 0.13ex\vrule height 0.13ex depth 0pt width 1.3ex}}

\DeclareMathOperator{\dist}{dist}		
\DeclareMathOperator{\spt}{spt}		
\DeclareMathOperator{\supp}{supp} 		

\DeclareMathOperator{\tr}{tr}		
\DeclareMathOperator*{\esssup}{ess\,sup}	

\title[On Singularities of Mean Curvature Flows with $H$ Bounds]{On the Structure of Singularities of Weak Mean Curvature Flows with Mean Curvature Bounds}
\author{Maxwell Stolarski}
\address{Warwick Mathematics Institute, University of Warwick}
\email{Max.Stolarski@warwick.ac.uk}
\urladdr{https://homepages.warwick.ac.uk/~u2175999/}

\begin{document}

\maketitle 

\begin{abstract}
	This paper studies singularities of mean curvature flows with integral mean curvature bounds $H \in L^\infty L^p_{loc}$ for some $p \in ( n, \infty]$.
	For such flows, any tangent flow is given by the flow of a stationary cone $\mathbf{C}$.
	When $p = \infty$ and $\C$ is a regular cone, we prove that the tangent flow is unique.
	These results hold for general integral Brakke flows of arbitrary codimension in an open subset $U \subset \mathbb{R}^N$ with $H \in L^\infty L^p_{loc}$.
	For smooth, codimension one mean curvature flows with $H \in L^\infty L^\infty_{loc}$, we also show that, at points where a tangent flow is given by an area-minimizing Simons cone, there is an accompanying limit flow given by a smooth Hardt-Simon minimal surface.
\end{abstract}

\tableofcontents

\section{Introduction}

A time-dependent family of embeddings $F : M^n \times [0, T) \to \R^N$ is said to evolve by mean curvature flow if
	$$\partial_t F =  H	$$
where $H = \vec H_{M_t}$ denotes the mean curvature vector of the embedded submanifold $M_t = F( M \times \{ t \} ) \subset \R^N$.
Submanifolds evolving by mean curvature flow often develop singularities in finite time $T < \infty$.
Huisken \cite{Huisken84} showed that the second fundamental form $A$ of a compact hypersurface $M_t$ evolving by mean curvature flow always blows up at a finite-time singularity $T < \infty$, that is,
	$$\limsup_{t \nearrow T} \sup_{x \in M_t} |  A | = \infty.$$

Given Huisken's result \cite{Huisken84}, it is natural to ask if the trace of the second fundamental form, namely the mean curvature $ H = \tr  A$, must also blow up at finite-time singularities of the mean curvature flow.
\cite{Stolarski23} answered this question in the negative by showing the mean curvature flow solutions constructed by Vel{\'a}zquez \cite{V94} develop finite-time singularities even though $H$ remains uniformly bounded $\sup_{t \in [0, T)} \sup_{x \in M_t } | H| < \infty$.
In dimension $n=7$, \cite{ADS21} further showed how to extend these mean curvature flow solutions $M_t^7 \subset \R^8$ to weak mean curvature flows defined for later times $t \in [0, T + \epsilon)$ 
in such a way that $H$ remains uniformly bounded and the flow has an isolated singularity at $(\mathbf 0 , T) \in \R^{8} \times [0, T+\epsilon)$.

Given that there exist smooth mean curvature flows which develop singularities with bounded mean curvature \cite{V94, Stolarski23, ADS21},
the focus of this current article is to instead study the singularities of \emph{any} mean curvature flow $M^n_t \subset \R^N$ with uniform mean curvature bounds.
In this general setting, uniform mean curvature bounds along the flow allow us to incorporate the well-developed theory of varifolds with bounds on their first variation.
In particular, we leverage that theory to obtain the following result which holds more generally for weak mean curvature flows of arbitrary codimension with integral mean curvature bounds:

\begin{thm} \label{meta thm 1}
	Let $2 \le n < N$ and let $U \subset \R^N$ be open.
	Let $(\mu_t)_{t \in (a,b)}$ be an integral $n$-dimensional Brakke flow in $U \subset \R^N$ with locally uniformly bounded areas and generalized mean curvature $H \in L^\infty L^p_{loc}(U \times (a,b))$ for some $p \in (n, \infty]$.
	
	Except for a countable set of times $t \in (a,b)$, $(\mu_t)_{t \in (a,b)}$ equals the Brakke flow of a family of integer rectifiable $n$-varifolds $(V_t)_{t \in (a,b]}$ that extends to the final time-slice $t = b$.
	
	For any $(x_0, t_0) \in U \times (a, b]$, any tangent flow of $(\mu_t)_{t \in (a,b)}$ at $(x_0, t_0)$ is given by the static flow of a stationary cone $\C$.
	The stationary cones $\C$ that arise as tangent flows to $(\mu_t)_{t \in (a,b)}$ at $(x_0, t_0)$ are exactly the tangent cones to $V_{t_0}$ at $x_0$.
\end{thm}

\noindent Note that the stationary cones $\C$ in Theorem \ref{meta thm 1} are generally integer rectifiable $n$-varifolds that are dilation invariant and have zero first variation (see Lemma \ref{lem nice representative} and Theorem \ref{thm blow-ups} for more precise statements).

Given a Brakke flow $(\mu_t)$, tangent flows at $(x_0,t_0)$ are, simply-speaking, subsequential limits of parabolic rescalings of $(\mu_t)$ based at $(x_0, t_0)$.
Analogously, tangent cones of a varifold $V_{t_0}$ at $x_0$ are subsequential limits of spatial rescalings of $V_{t_0}$ based at $x_0$.
Thanks to suitable monotonicity formulas and compactness theorems, tangent flows and tangent cones always exist.
However, the uniqueness of tangent cones and tangent flows, that is independence of subsequence, is generally an open problem.
This uniqueness question is fundamental for singularity analysis and regularity.

Theorem \ref{meta thm 1} in particular shows that, for flows with mean curvature $H \in L^\infty L^p_{loc}( U \times (a,b))$, 
uniqueness of the tangent \emph{flow} of $(\mu_t)$ at $(x_0, t_0)$ is equivalent to uniqueness of the tangent \emph{cone} of $V_{t_0}$ at $x_0$.
Because of this correspondence, we are able to prove the following uniqueness result:
\begin{thm} \label{meta thm 2}
	Let $2 \le n < N$ and let $U \subset \R^N$ be open.
	Let $(\mu_t)_{t \in (a,b)}$ be an integral $n$-dimensional Brakke flow in $U \subset \R^N$ with locally uniformly bounded areas and generalized mean curvature $H \in L^\infty L^\infty_{loc}(U \times (a,b))$.
	
	If a tangent flow to $(\mu_t)_{t \in (a,b)}$ at $(x_0, t_0) \in U \times (a,b]$ is given by the static flow of a regular cone $\C$ (with multiplicity one),
	then this is the unique tangent flow to $(\mu_t)$ at $(x_0, t_0)$.
\end{thm}

It is worth mentioning some related uniqueness results.
Huisken's monotonicity formula ensures tangent flows are always self-similarly shrinking $\Sigma_{t} = \sqrt{-t} \Sigma_{-1}$ for $t < 0$.
There are various uniqueness results depending on $\Sigma = \Sigma_{-1}$.
When $\Sigma$ is compact, \cite{Schulze14} proved that if $\sqrt{-t} \Sigma$ is a tangent flow of a mean curvature flow $(M_t)$ at $(x_0,t_0)$, then it's the unique tangent flow at $(x_0, t_0)$.
\cite{CM15} proved uniqueness of the tangent flow $\sqrt{-t}\Sigma$ when $\Sigma = \R^{n-k} \times \Sph^k$ is a generalized cylinder.
\cite{CS21} proved uniqueness of the tangent flow when $\Sigma$ is smooth and asymptotically conical.
\cite{Zhu20,LZ24} obtained additional generalizations of the uniqueness results in \cite{CM15, CS21}, respectively.

Importantly, the uniqueness results in \cite{CS21, LZ24} require $\Sigma$ to be smooth and do not apply when $\Sigma = \C$ is a minimal cone for example.
The first uniqueness result for tangent flows given by non-smooth $\Sigma$ came in \cite{LSS22} which showed that,
for 2-dimensional Lagrangian mean curvature flows $L^2_t \subset \mathbb{C}^2$, tangent flows given by a transverse pair of planes $\Sigma = P_1 \cup P_2 \subset \mathbb{C}^2$ are unique.
The uniqueness result Theorem \ref{meta thm 2} here applies to non-smooth minimal cones $\Sigma = \C$ which arise as tangent flows to general integral Brakke flows of any codimension in an open subset $U \subset \R^N$, albeit under the assumption of a uniform mean curvature bound $H \in L^\infty L^\infty_{loc}$.

Under additional hypotheses, we can also describe an accompanying limit flow that arises as a more general blow-up limit around a singularity.
\begin{thm}\label{meta thm 3}
	Let $\cl{M} = (M_t^n)_{t \in (a,b)}$ be a smooth, properly embedded mean curvature flow in an open subset $U \subset \R^{n+1}$ with mean curvature $H \in L^\infty L^\infty_{loc}( U \times (a,b) )$.
	
	If the tangent flow of $\cl{M}$ at $(x, b) \in U \times \{ b \}$ is given by the static flow of a generalized Simons cone $\C^n \subset \R^{n+1}$ (with multiplicity one) and $\C$ is area minimizing,
	then there exists a sequence $(x_i, t_i ) \in U \times (a,b)$ with $\lim_{i \to \infty} (x_i , t_i) = (x, b)$ and 
	a sequence $\lambda_i \searrow 0$ such that the sequence of rescaled mean curvature flows
	$\cl{M}_i = \cl{D}_{\lambda_i^{-1}} ( \cl{M} - (x_i , t_i) )$
	converges to the static flow of a smooth Hardt-Simon minimal surface.
\end{thm}
We refer the reader to Section \ref{sect Pinching Hardt-Simon Minimal Surfs} for the definitions relevant to Theorem \ref{meta thm 3}.
For now, we simply note that the mean curvature flow solutions constructed in \cite{V94} provide examples of mean curvature flows satisfying Theorem \ref{meta thm 3}.
Theorem \ref{meta thm 3} states that general mean curvature flows with $H \in L^\infty L^\infty_{loc}$ in some sense mimic the dynamics of Vel{\'a}zquez's mean curvature flow solutions \cite{V94} near Simons cone singularities.

The paper is organized as follows: 
In Section \ref{sect Preliminaries}, we establish notation, specify definitions, and obtain some general results used throughout the paper. 
In particular, Theorem \ref{meta thm 1} is proven here.
Section \ref{sect Uniqueness of Tangent Flows} proves the uniqueness of tangent flows given by regular stationary cones, Theorem \ref{meta thm 2}.
In Section \ref{sect Pinching Hardt-Simon Minimal Surfs}, we obtain refined dynamics of mean curvature flows near regular stationary cones and prove Theorem \ref{meta thm 3}.
Finally, Appendix \ref{appendix Varifolds with H in L^p} reviews some well-known results about integral varifolds with generalized mean curvature $H \in L^p_{loc}$ that are cited throughout the paper.

\bigskip \noindent \textbf{Acknowledgements.} I would like to thank Professor Felix Schulze for many helpful conversations, particularly regarding the results in \cite{LSS22}.

The author is supported by a Leverhulme Trust Early Career Fellowship (ECF-2023-182).

\section{Preliminaries}		\label{sect Preliminaries}

\subsection{Brakke Flows with $H$ Bounds}

\begin{definition} \label{defn varifold with H in L^p}
	Let $2 \le n < N$, $U \subset \R^N$ be open, and $p \in (1, \infty]$.
	Let $V$ be an integer rectifiable $n$-varifold in $U$ and $\mu_V$ be the associated measure on $U$.
	We say that $V$ has \emph{generalized mean curvature $H \in L^p_{loc}(U)$} if
	there exists a Borel function $H : U \to \R^N$ with $H \in L^p_{loc}( U, d\mu_V)$ such that the first variation $\delta V$ satisfies
	\begin{gather*}
		\delta V ( X )= - \int H \cdot X d \mu_{V} \qquad \forall X \in C^1_c ( U, \R^N) .
	\end{gather*}
\end{definition}

\begin{remark}
	Observe that, by the definition given in Definition \ref{defn varifold with H in L^p}, if $V$ has generalized mean curvature $H \in L^p_{loc}(U)$ then $V$ has no generalized boundary in $U$.
	This convention differs somewhat from the existing literature but we adopt it nonetheless to simplify the statements in the remainder of the paper.
\end{remark}

\begin{definition}
	Let $2 \le n < N$, $U \subset \R^N$ be open, and $p \in (1, \infty]$.
	Let $(\mu_t)_{t \in (a,b)}$ be an $n$-dimensional integral Brakke flow in $U$.
	Recall that, for a.e. $t \in (a,b)$, there exists an integer rectifiable $n$-varifold $V_t$ such that $\mu_{V_t} = \mu_t$.
	
	We say the Brakke flow $(\mu_t)_{t \in (a,b)}$ has \emph{generalized mean curvature $H \in L^\infty L^p_{loc}(U \times (a,b) )$}
	if 
	for a.e. $t \in (a,b)$ there exists an integer rectifiable $n$-varifold $V_t$ with generalized mean curvature $H_t \in L^p_{loc}(U)$
	such that $\mu_{V_t} = \mu_t$ and 
 	$$\| H \|_{L^\infty L^p ( K \times (a,b) ) } \doteqdot \esssup_{t \in (a, b)} \| H_t \|_{L^p( K, d\mu_t)} < \infty
	\qquad ( \forall K \Subset U)\footnotemark.$$
\end{definition}

\footnotetext{Throughout, ``$K \Subset U$" means $K \subseteq U$ and $K$ is compact.}

\begin{assumption} \label{Main Assumption-}
Let $2 \le n < N$, $U \subset \R^N$ be an open subset, $-\infty < a < b < \infty$.
Throughout, we assume $(\mu_t)_{t \in (a, b) }$ is an $n$-dimensional integral Brakke flow in $U$
such that $(\mu_t)$ has locally uniformly bounded areas, that is, 
		$$\sup_{t \in (a, b)} \mu_t( K) <\infty		\qquad \forall K \Subset U.$$
For simplicity, we will abbreviate this assumption as ``$(\mu_t)_{t \in (a, b)}$ is an integral $n$-Brakke flow in $U \subset \R^N$ with locally uniformly bounded areas" in the remainder of the paper. 		
		
\end{assumption}

We will also often assume that, for some $p \in (1, \infty]$,
\begin{equation*}
	(\mu_t)_{t \in (a,b)} \text{ has generalized mean curvature } H \in L^\infty L^p_{loc}( U \times ( a, b) ) ,
\end{equation*}
but this assumption will be indicated in each statement.

\begin{remark}
	The assumption that the flow has locally uniformly bounded areas is quite mild.
	For example, it holds for Brakke flows $(\mu_t)_{t \in [a,b)}$ starting from initial data $\mu_a$ with locally bounded areas, i.e. $\mu_a(K) < \infty$ for all $K \Subset U$.
	Indeed, this can be seen by using Brakke's inequality with suitably defined spherically shrinking test functions.

	On the other hand, the assumption that the flow has generalized mean curvature $H \in L^\infty L^p_{loc}( U \times (a,b))$ is much more restrictive.
	Nonetheless, \cite{Stolarski23} shows that even smooth mean curvature flows $(M_t^n \subset \R^{n+1})_{t \in (a,b)}$ with $H \in L^\infty L^\infty ( \R^{n+1} \times (a,b))$ can develop singularities at the final time $t = b$.
	Combined with the work of \cite{ADS21}, there are non-smooth Brakke flows with $H \in L^\infty L^\infty (\R^N \times (a,b))$ with mild singularities and small singular sets, informally speaking.
\end{remark}

The next lemma shows that Brakke flows with $H \in L^\infty L^p_{loc}$ can be changed at countably many times to get a Brakke flow which is a varifold with $H \in L^p_{loc}$ at \emph{every} time. 
Moreover, the flow naturally extends to the final time-slice.

\begin{lem} \label{lem nice representative}
	Let $(\mu_t)_{t \in (a,b)}$ be an integral $n$-Brakke flow in $U \subset \R^N$ with locally uniformly bounded areas and generalized mean curvature $H \in L^\infty L^p_{loc} (U \times (a,b) )$ for some $p \in (1, \infty]$.
	
	Then for all $t \in (a, b]$, there exists a unique integer rectifiable $n$-varifold $V_t$ with generalized mean curvature in $H_{V_t} \in L^p_{loc}(U)$ such that:
	\begin{enumerate}
		\item $\mu_t \le  \lim_{t' \nearrow t} \mu_t = \mu_{V_t}$ for all $t \in (a, b)$,
		\item $\mu_t = \mu_{V_t}$ for all but countably many $t \in (a, b)$,
		\item for all $t \in (a, b]$ and all $K \Subset U$
			$$\mu_{V_t} (K) \le \sup_{\tau \in (a, b)} \mu_\tau (K)
			\qquad \text{and} \qquad 
			\ \| H_{V_t} \|_{L^p (K ) } \le \| H \|_{L^\infty L^p ( K \times (a, b) )}$$
		\item for all $t \in (a, b]$
			$$\lim_{t' \nearrow t} V_{t'}(f) = V_t(f)		\qquad \forall f \in C^0_c ( G(n, U) )$$
		and	
		\item $(\mu_{V_t} )_{t \in (a, b] }$ is an $n$-dimensional integral Brakke flow in $U$.
	\end{enumerate}	
\end{lem}
\begin{proof}
	Let $t \in (a, b]$.
	Since $(\mu_t)$ has $H \in L^\infty L^p_{loc}(U \times (a, b))$, 
	there exists a sequence of times $t_j \nearrow t$ (with $t_j < t$) and integer rectifiable $n$-varifolds $\tilde V_j$ with generalized mean curvature $H_{\tilde V_j}$ in $L^p_{loc}(U)$ such that 
	$\mu_{t_j} = \mu_{\tilde V_j}$ and 
	$$\| H_{\tilde V_j} \|_{L^p(K, d\mu_{t_j} ) } \le \| H \|_{L^\infty L^p(K \times (a, b))} \doteqdot C_K < \infty		\qquad \forall K \Subset U.$$
	By compactness Lemma \ref{Lem Compactness},
	there exists a subsequence (still denoted $\tilde V_j$) and an integer rectifiable $n$-varifold $V_t$ 
	with locally bounded areas and generalized mean curvature $H_{V_t} \in L^p_{loc}(U)$
	such that 
	$\tilde V_j \rightharpoonup V_t$ as varifolds and 
		$$\mu_{V_t}(K) \le \sup_{\tau \in (a,b)} \mu_\tau(K) \qquad \text{and} \qquad \| H_{V_t} \|_{L^\infty(K, d\mu_{V_t} )} \le C_K		\qquad \forall K \Subset U.$$
	This defines the varifold $V_t$ for any $t \in (a, b]$ and proves it satisfies (3).
	
	By \cite[7.2(ii)]{Ilmanen93}, 
	for any $f \in C^0_c ( U, \R_{\ge 0} )$ and any $t \in (a, b)$
	\begin{equation} \label{inequality (1) proof}
		\mu_t ( f ) \le \lim_{s \nearrow t} \mu_s (f) = \lim_{j \to \infty} \mu_{\tilde V_j}(f) = \mu_{V_t}(f).
	\end{equation}
	This proves (1).
	For (2), simply note that \cite[7.2(iii)]{Ilmanen93} implies the first inequality in \eqref{inequality (1) proof} is an equality for all but countably many times $t \in (a, b)$.
	Note additionally that the equality $\lim_{s \nearrow t} \mu_s = \mu_{V_t} $ in (1) implies the varifold $V_t$ is unique.
	
	To prove (4), let $t \in (a, b]$ and $f \in C^\infty_c(U, \R_{\ge 0})$.
	Take a sequence $t_j \nearrow t$.
	By \cite[7.2(i)]{Ilmanen93}, there exists $C_f$ such that $\mu_t(f) - C_f t$ is decreasing in $t$.
	It follows that, for any $j$, 
	\begin{align*}
		\mu_{V_t} (f) 
		&= \lim_{s \nearrow t}( \mu_s (f) - C_f s ) + C_f t	&& (1)\\
		&\le \mu_{t_j} (f) - C_f t_j + C_f t	\\
		&\le \mu_{V_{t_j}} (f) + C_f ( t - t_j )		&& (1).
	\end{align*}
	Taking $j \to \infty$ gives $\mu_{V_t} (f) \le \liminf_j \mu_{V_{t_j}}(f)$.
	For the reverse inequality, let $\epsilon > 0$.
	Recall $\mu_{V_t} (f) = \lim_{s \nearrow t} \mu_t(f)$ and similarly for the the $V_{t_j}$.
	Thus, there exists $s = s(\epsilon) < t$ such that $| s - t | < \epsilon $ and $| \mu_{V_t}(f) - \mu_s (f) | < \epsilon.$
	There exists $J$ such that $s < t_j < t$ for all $j > J$.
	It follows that, for $j > J$, 
	\begin{align*}
		\mu_{V_{t_j}}(f) - \mu_{V_t}(f)
		&\le \mu_{V_{t_j}}(f) - \mu_{s}(f) + \epsilon \\
		&= \lim_{\sigma \nearrow t_j} ( \mu_\sigma (f) - C_f \sigma ) + C_f t_j - \mu_{s}(f) + \epsilon	\\
		&\le \mu_s(f) - C_f s + C_f t_j - \mu_{s} (f) + \epsilon	\\
		&= C_f ( t_j - s ) + \epsilon.
	\end{align*}
	Taking $j \to \infty$ then $\epsilon \to 0$ gives $\limsup_j \mu_{V_{t_j} }(f) \le \mu_{V_t}(f)$.
	Since the sequence $t_j \nearrow t$ and $f \in C^\infty_c(U, \R_{\ge 0} )$ were arbitrary, it follows that
		$$\lim_{t' \nearrow t} \mu_{V_{t'} }(f) = \mu_{V_t} (f)	\qquad \forall f \in C^\infty_c ( U, \R_{\ge 0} ).$$
	Convergence as varifolds then follows from Lemma \ref{Lem Checking Convergence} and completes the proof of (4).
	
	To prove (5), it suffices to check Brakke's inequality.
	Let $a < t_0 < t_1 \le b$ and $f \in C^1_c ( U \times [t_0, t_1] )$ with $f \ge 0$.
	Then
	\begin{align*}
		\mu_{V_{t_1} } (f) - \mu_{V_{t_0}} (f)
		&\le \lim_{s \nearrow t_1} \mu_s (f) - \mu_{t_0} (f)		&& (1)	\\
		&\le \lim_{s \nearrow t_1} \int_{t_0}^s \int \partial_t f + \nabla f \cdot H - |H|^2 f d\mu_t dt	\\
		&= \int_{t_0}^{t_1} \int\partial_t f + \nabla f \cdot H - |H|^2 f d\mu_{V_t} dt	&& (2).
	\end{align*}
\end{proof}

\subsection{Huisken's Monotonicity Formula and Gaussian Density} \label{subsect Huisken's mono formula}

Let
	$$\Phi_{x_0, t_0} (x, t ) \doteqdot \frac{1}{ \left( 4 \pi ( t_0 - t) \right)^{n/2} } e^{ - \frac{ | x - x_0|^2}{4 ( t_0 - t) }}
	\qquad ( t< t_0)$$
denote the backwards heat kernel based at $(x_0, t_0)$.
Let 
	$$\phi_{x_0, t_0; r} (x, t) \doteqdot \left( 1 - \frac{ | x - x_0|^2 - 2n ( t_0 - t) }{ r } \right)^3_+
	\qquad ( t \le t_0 )$$
denote the spherically shrinking localization function based at $(x_0, t_0)$ with scale $r > 0$.

Huisken's monotonicity formula and its localized analogue for Brakke flows states
\begin{multline} \label{Huisken's mono forla}
	\int \Phi_{x_0, t_0}( \cdot  , t) \phi_{x_0, t_0; r} ( \cdot , t) d \mu_t 
	- \int \Phi_{x_0, t_0}( \cdot  , s) \phi_{x_0, t_0; r} ( \cdot, s) d \mu_s \\
	\le - \int_s^t \int \left| H + \frac{ ( x - x_0 )^\perp}{2 ( t_0 - \tau ) } \right|^2 \Phi_{x_0, t_0} \phi_{x_0, t_0; r} d \mu_\tau d\tau \le 0
\end{multline}
for all $s < t < t_0$ and $r > 0$ such that $(\mu_t)$ is a Brakke flow on $B_{r + 2n (t_0 - s) } (x_0) \times [s, t]$.
Moreover, for Brakke flows $( \mu_t)_{t \in (a,b)}$ defined in $U \subset \R^N$ and points $x_0 \in B_{2 r_0} (x_0) \subset U$, the Gaussian density
\begin{equation} \label{Gaussian density defn}
	\Theta_\mu ( x_0, t_0 ) \doteqdot \lim_{t \nearrow t_0} \int \Phi_{x_0, t_0} ( \cdot, t) \phi_{x_0, t_0 ; r} ( \cdot , t) d \mu_t
\end{equation}
is well-defined for $t_0 \in (a, b]$ and independent of $r \in (0 , r_0)$.

For an $n$-dimensional Brakke flow $(\mu_t)$ and $\lambda > 0$, define the parabolically dilated Brakke flow $\cl{D}_{x_0, t_0; \lambda} \mu$  based at $(x_0, t_0)$ to be
	$$\left[ \cl{D}_{x_0, t_0; \lambda}  \mu \right]_t ( A) \doteqdot 
	\lambda^n \mu_{t_0 + t \lambda^{-2} } ( \lambda^{-1} A  + x_0).$$
Often, we omit the basepoint $x_0, t_0$ and simply write $\cl{D}_\lambda \mu$ when the basepoint is clear from context.

\subsection{Equivalent Notions of Density}

Recall the density of varifold $V$ at $x_0 \in \R^N$ is given by
	$$\theta_V(x_0) \doteqdot \lim_{\rho \searrow 0} \frac{ \mu_V( B_\rho(x_0))}{\omega_n \rho^n}$$
when the limit exists and where $\omega_n$ denotes the volume of the unit $n$-ball.
In the remainder of the article, we use $\eta_{x_0, \lambda}$ to denote the spatial translation and dilation
	$$\eta_{x_0, \lambda} : \R^N \to \R^N, \qquad \eta_{x_0, \lambda} (x) \doteqdot \lambda ( x - x_0 ).$$
$\eta_{x_0, \lambda}$ naturally induces a map of integer rectifiable $n$-varifolds which we denote by $(\eta_{x_0, \lambda})_\sharp$.
Specifically, if $V = (\Gamma, \theta)$, then $(\eta_{x_0, \lambda})_\sharp V = ( \eta_{x_0, \lambda} (\Gamma) , \theta \circ \eta_{x_0,\lambda}^{-1} )$.
When the basepoint $x_0$ is clear from context or $x_0 =  0$, we shall often simply write $\eta_{\lambda}$ for simplicity.

When $H \in L^\infty L^p_{loc}$ with $p > n$, the monotonicity formula \eqref{simple mono forla eqn} for varifolds gives a characterization of subsequential blow-ups of time-slices.

\begin{lem}[Existence of Tangent Cones of Time-Slices] \label{lem tangent cone}
	Let $(\mu_t)_{t \in (a,b)}$ be an integral $n$-Brakke flow in $U \subset \R^N$ with locally uniformly bounded areas (Assumption \ref{Main Assumption-}) and 
	generalized mean curvature $H  \in L^\infty L^p_{loc}(U \times (a,b))$ for some $p \in (n, \infty]$. 
	Let $(V_t)_{t \in (a, b] }$ be the family of varifolds as in Lemma \ref{lem nice representative}.
	
	For any $(x_0, t_0) \in U \times (a, b]$ and any sequence $\lambda_i \nearrow +\infty$,
	there exists a subsequence (still denoted $\lambda_i$) 
	and an integer rectifiable $n$-varifold $\C$ in $\R^N$ such that
		$$\left( \eta_{ x_0, \lambda_i} \right)_{\sharp} V_{t_0} 
		\rightharpoonup \C \qquad \text{as varifolds in } \R^N.$$
	Moreover, $\C$ is a stationary, dilation invariant varifold in $\R^N$ with
		$$\frac{ \mu_{\C} ( B_r ( 0 ) ) }{ \omega_n r^n } = \theta_{V_{t_0}} (x_0) 
		= \lim_{\rho \searrow 0} \frac{ \mu_{V_{t_0} } ( B_\rho ( x_0 ) ) }{ \omega_n \rho^n }
		\qquad \forall \,  0 < r < \infty.$$
		
	Any $\C$ that arises as such a limit is called a \emph{tangent cone} of $V_{t_0}$ at $x_0$.	
\end{lem}
\begin{proof}
	By Lemma \ref{lem nice representative}, $V_{t_0}$ has generalized mean curvature $H = H_{V_{t_0}} \in L^p_{loc}( U)$.
	It follows that the dilations $V_i \doteqdot (\eta_{x_0, \lambda_i})_{\sharp} V_{t_0}$ have generalized mean curvature $ H_i \in L^p_{loc}( \lambda_i (U - x_0) )$ with
	\begin{equation} \label{H_i decay est}
		\| H_i \|_{L^p(K , d \mu_{V_i} ) } \le \lambda_i^{\frac{n}{p} - 1 } \| H_{V_{t_0}} \|_{L^p ( \ol{B}_R(x_0), d \mu_{V_{t_0} } ) }
	\end{equation}
	for all $K \subset \R^N$ compact and $i \gg 1$ sufficiently large so that $\frac{1}{\lambda_i} K + x_0 \subset \ol{B}_R(x_0) \Subset U$.
	In particular, $\lim_{i \to \infty} \| H_i \|_{L^p(K , d \mu_{V_i} ) } = 0$.
	
	Since $p > n$, there is the monotonicity formula \eqref{simple mono forla eqn} for integral $n$-varifolds with $H \in L^p_{loc}$ (Proposition \ref{prop simple mono forla}, see also \cite[Ch. 4, \S 4]{SimonGMT}).
	In particular, the density 
		$$\theta_{V_{t_0}} (x_0) \doteqdot \lim_{\rho \searrow 0} \frac{ \mu_{V_{t_0}} (B_\rho(x_0) )}{ \omega_n \rho^n}$$
	is well-defined
	and the blow-up sequence $V_i$ has uniform local area bounds.
	
	It now follows from compactness (Lemma \ref{Lem Compactness})
	that there is a subsequence (still denoted $V_i$) and an integer rectifiable $n$-varifold $\C$ in $\R^N$ such that $V_i \rightharpoonup \C$.
	By \eqref{H_i decay est} and Lemma \ref{Lem Compactness},
	$$\| H_\C \|_{L^p(K, d\mu_\C )} \le \limsup_{i } \| H_{V_i} \|_{L^p (K, d \mu_{V_i} ) } = 0,$$
	which implies $\C$ is stationary.
	Moreover, the convergence $V_i \rightharpoonup \C$ implies
		$$\frac{ \mu_\C ( B_r(0 ) )}{ \omega_n r^n } = \theta_{V_{t_0} } (x_0)	\qquad \forall 0 < r < \infty.$$
	The monotonicity formula \eqref{simple mono forla eqn} for stationary varifolds finally implies $\C$ is dilation invariant, that is
		$$( \eta_{0, \lambda} )_\sharp \C = \C 		\qquad \forall 0 < \lambda <\infty$$
	 (see e.g. \cite[Ch. 8, \S 5]{SimonGMT}).
\end{proof}

We will show in Lemma \ref{lem tangent flow} below that Huisken's monotonicity formula \eqref{Huisken's mono forla} similarly allows us to extract tangent \emph{flows} from \emph{space-time} blow-ups.
First, we show local uniform area \emph{ratio} bounds follow from the local uniform area bounds.

\begin{lem}[Area Ratio Bounds] \label{lem area ratio bound}
	Let $(\mu_t)_{t \in (a,b)}$ be an integral $n$-Brakke flow in $U \subset \R^N$ with locally uniformly bounded areas (Assumption \ref{Main Assumption-}) 
	and generalized mean curvature $H \in L^\infty L^p_{loc}(U \times(a,b))$ for some $p \in (n, \infty]$.
	
	For any $K \Subset U$, there exists $C$ such that
		$$\sup_{t \in ( a , b) } \sup_{B_r(x) \subset K } \frac{ \mu_t ( B_r(x)) }{r^n} \le C.$$
\end{lem}
\begin{proof}
	The proof will proceed by combining the monotonicity formula \eqref{simple mono forla eqn} with the local uniform area bounds.
	Given $K \Subset U$, there exists $\epsilon > 0$ such that
		$$K' \doteqdot \ol{B_\epsilon (K) }
		= \ol{ \bigcup_{x \in K } B_\epsilon (x) } \Subset U.$$
	Let
		$$R_0 \doteqdot \sup \{ r > 0 : B_r(x) \subset K \text{ for some } x \},
		\qquad R_1 \doteqdot \max \{ R_0, 1 \}.$$
	For any $B_r(x) \subset K$,
	there exists $R \in [\max \{ \epsilon, r \} , R_1]$ such that $B_r (x) \subset B_R(x) \subset K'$.
	Let $(V_t)_{t \in (a, b]}$ be the family of varifolds as in Lemma \ref{lem nice representative} and let $t \in (a,b)$.
	The monotonicity formula \eqref{simple mono forla eqn} implies
	$$ \frac{ \mu_{V_t} ( B_r(x) ) }{ r^n}
	\le \frac{ \mu_{V_t} ( B_r(x) ) }{ r^n} e^{ \frac{ \| H_{V_t} \| }{ 1 - n/ p } r } + e^{ \frac{ \| H_{V_t} \| }{ 1 - n/ p } r }
	\le \frac{ \mu_{V_t} ( B_R(x) ) }{ R^n} e^{ \frac{ \| H_{V_t} \| }{ 1 - n/ p } R } + e^{ \frac{ \| H_{V_t} \| }{ 1 - n/ p } R }$$
	where $\| H_{V_t} \| = \| H_{V_t} \|_{L^p ( B_R (x) , d\mu_{V_t} ) }$.
	It then follows that 
	\begin{align*}
		\frac{ \mu_t (B_r (x) ) }{ r^n}
		&\le \frac{ \mu_{V_t} (B_r (x) ) }{ r^n}	\\
		&\le \frac{ \mu_{V_t} ( B_R(x) ) }{ R^n} e^{ \frac{ \| H_{V_t} \|_{L^p ( K' )} }{ 1 - n/ p } R } + e^{ \frac{ \| H \|_{L^p(K') }}{ 1 - n/ p } R } \\
		&\le \frac{ \mu_{V_t} ( K' ) }{ \epsilon^n} e^{ \frac{ \| H_{V_t} \|_{L^p ( K' )} }{ 1 - n/ p } R_1 } + e^{ \frac{ \| H \|_{L^p(K') }}{ 1 - n/ p } R_1 } \\
		&\le  \frac{ \sup_{\tau \in (a,b)} \mu_{\tau} ( K' ) }{ \epsilon^n} e^{ \frac{ \| H \|_{L^\infty L^p ( K' \times ( a, b)  )} }{ 1 - n/ p } R_1 } + e^{ \frac{ \| H \|_{L^\infty L^p(K' \times ( a, b) ) }}{ 1 - n/ p } R_1 } < \infty.
	\end{align*}
	Taking the supremum over $t \in (a, b)$ and $B_r(x) \subset K$ completes the proof.
\end{proof}

\begin{lem}[Existence of Tangent Flows] \label{lem tangent flow}
	Let $(\mu_t)_{t \in (a, b)}$ be an integral $n$-Brakke flow in $U \subset \R^N$ with locally uniformly bounded areas and generalized mean curvature $H \in L^\infty L^p_{loc}(U \times (a,b))$ for some $p \in (n, \infty]$.

	For any $(x_0, t_0) \in U \times (a,b]$ and any sequence $\lambda_i \nearrow +\infty$,
	there exists a subsequence (still denoted $\lambda_i$) and 
	an integer rectifiable $n$-varifold $\C$ in $\R^N$ such that for any $t < 0$
		$$\left( \cl{D}_{x_0, t_0; \lambda_i} \mu \right)_t
			\rightharpoonup \mu_\C \qquad \text{as $i \to \infty$ (weakly as measures)}.$$

	Moreover, $\C$ is a stationary, dilation invariant varifold in $\R^N$ with 
		$$\Theta_\mu ( x_0, t_0 ) = \frac{1}{( 4 \pi)^{n/2}} \int e^{- \frac{|x|^2}{4} } d \mu_\C .$$	
		
	Any $\C$ that arises as such a limit is called a \emph{tangent flow} of $(\mu_t)_{t \in (a,b)}$ at $(x_0, t_0)$.	
\end{lem}

\begin{proof}
	Fix $(x_0, t_0) \in U \times ( a, b]$.
	By Lemma \ref{lem area ratio bound}, $(\mu_t)$ has local uniform area ratio bounds in a neighborhood of $(x_0, t_0)$.
	It follows from Huisken's monotonicity formula \eqref{Huisken's mono forla} and the compactness of Brakke flows with local uniform area bounds that there exists a Brakke flow $(\mu^\infty_t)_{t < 0}$ on $\R^N$ 
	and a subsequence (still denoted $\lambda_i$) such that
		$$\left( \cl{D}_{\lambda_i} \mu \right)_t 
		= \left( \cl{D}_{x_0, t_0; \lambda_i} \mu \right)_t 
		\rightharpoonup \mu^\infty_t
		\qquad \forall t < 0$$ 
	(see \cite[Lemma 8]{IlmanenSurfs} or \cite{White94}).
	Moreover, $\mu^\infty$ is a shrinker for $t < 0$ in the sense that
	\begin{equation} \label{shrinker eqn 1}
		\mu_t ( A) = (-t)^{n/2} \mu_{-1} ( A / \sqrt{ -t} )	\qquad  \forall t < 0
	\end{equation}
	and, for any $t < 0$, $\mu^\infty_{t}$ is an integer rectifiable $n$-varifold with
	\begin{equation} \label{shrinker eqn 2}
		H + \frac{ x^\perp}{2(-t)} = 0		\qquad \text{for } \mu^\infty_{t}\text{-a.e. } x \in \R^N.
	\end{equation}
	Additionally,
		$$\Theta_\mu( x_0, t_0) = \frac{1}{ ( 4 \pi (-t) )^{n/2} } \int e^{- \frac{|x|^2}{4 (-t)} } d \mu^\infty_t
		\qquad  \forall t < 0.$$
	
	Because $(\mu_t)$ is a Brakke flow with $H \in L^\infty L^p_{loc}$,
	there exists $t < 0$ such that, for all $i$,
	$( \cl{D}_{\lambda_i} \mu )_t$ is represented by an integer rectifiable $n$-varifold $V_i$ with $H_{V_i} \in L^p_{loc} ( \lambda_i U )$ and, for any $K \Subset \R^N$,
		$$\| H_{V_i} \|_{L^p( K , d \mu_{V_i} )}
		\le \lambda_i^{\frac{n}{p} - 1} \| H \|_{L^\infty L^p( B_{r_0} (x_0) \times (a,b) ) }$$
	for all $i \gg 1$ sufficiently large such that $\lambda_i^{-1} K + x_0 \subset B_{r_0} ( x_0)$ and $\ol{B_{r_0} (x_0) } \Subset U$. 
	Since $p > n$, 
		$$\lim_{i \to \infty} \| H_{V_i} \|_{L^p ( K , d \mu_{V_i} )} = 0
		\qquad \forall K \Subset \R^N.$$
	Since $\mu_{V_i} \rightharpoonup \mu^\infty_t$,
	the compactness of varifolds with mean curvature bounds Lemma \ref{Lem Compactness} and Lemma \ref{Lem Checking Convergence} imply that $\mu^\infty_t = \mu_\C$ for some integer rectifiable $n$-varifold $\C$ in $\R^N$ which is stationary ($H_\C = 0$).
	It then follows from \eqref{shrinker eqn 2} that $x^\perp = 0$ $\mu_\C$-a.e. and thus $\C$ is dilation invariant (see e.g. \cite[Ch. 8, \S 5]{SimonGMT}).
	Finally, \eqref{shrinker eqn 1} implies $\mu^\infty_t = \mu_\C$ for all $t < 0$.
	This completes the proof.
\end{proof}

A priori, the tangent \emph{cones} from Lemma \ref{lem tangent cone} could be entirely unrelated to the tangent \emph{flows} from Lemma \ref{lem tangent flow}.
Indeed, the two limiting objects arise from entirely different blow-up sequences.
The next theorem, however, proves that the tangent cones from Lemma \ref{lem tangent cone} exactly correspond to tangent flows from Lemma \ref{lem tangent flow}.

\begin{thm} \label{thm blow-ups}
	Let $(\mu_t)_{t \in (a, b)}$ be an integral $n$-Brakke flow in $U \subset \R^N$ with locally uniformly bounded areas and generalized mean curvature $H  \in L^\infty L^p_{loc}(U \times (a,b))$ for some $p \in (n, \infty]$.
	Let $(V_t)_{t \in (a, b] }$ be the family of varifolds as in Lemma \ref{lem nice representative}.
	
	For any $(x_0, t_0) \in U \times (a, b]$,
		$$\theta_{V_{t_0}} (x_0) = \Theta_\mu(x_0, t_0).$$
	For any sequence $\lambda_i \nearrow +\infty$, there exists a subsequence (still denoted $\lambda_i$) such that there are limiting integral $n$-varifolds $\C, \tilde \C$ as in Lemmas \ref{lem tangent cone}, \ref{lem tangent flow} respectively, and in fact $\C = \tilde \C$.
\end{thm}

\begin{proof}
	By translation, we can assume without loss of generality that $(x_0, t_0) = ( \mathbf 0 , 0) \in U \times (a,b]$.
	For any sequence $\lambda_i \nearrow +\infty$, denote the rescalings
		$$V^i_0 \doteqdot \left( \eta_{\lambda_i} \right)_{\sharp} V_{0},
		\qquad
		\mu^i_t \doteqdot (\cl{D}_{\lambda_i} \mu)_t.$$
	Lemmas \ref{lem tangent cone} and \ref{lem tangent flow} imply there exists a subsequence (still denoted by index $i$) such that 
		$$V^i_0 \rightharpoonup \C, \qquad \text{and} \qquad
		\mu^i_t \rightharpoonup \mu_{\tilde \C}	\quad (\forall t < 0),$$
	where $\C, \tilde \C$ are stationary, dilation invariant, integral $n$-varifolds.
	Additionally,
	\begin{gather*}
		\theta_{V_{0}} (\mathbf 0) = \frac{\mu_\C ( B_r(0))}{\omega_n r^n} \quad ( \forall 0 < r < \infty)
		\qquad \text{ and}	\\ 
		\Theta_\mu ( \mathbf 0, 0) = \frac{1}{ ( 4 \pi (-t) )^{n/2} } \int e^{- \frac{|x|^2}{4(-t)}} d \mu_{\tilde \C} 
		\quad ( \forall t < 0).
	\end{gather*}

	We first claim that $\mu_\C \le \mu_{\tilde \C}$.
	Let $f \in C^1_c( \R^N)$ with $f \ge 0$. Say $\supp f \subset B_R$.
	It follows that
	\begin{align*}
		\mu_\C (f)
		={}& \lim_{i \to \infty} \mu_{V^i_0} (f)	\\
		={}& \lim_{i \to \infty} \lim_{t' \nearrow 0} \mu^i_{t'} (f) && (\text{Lemma \ref{lem nice representative}})	\\
		\le{}& \lim_{i \to \infty} \lim_{t' \nearrow 0}\left( \mu^i_{-1} (f) + \int_{-1}^{t'} \int  H_i \cdot \nabla f - |H_i|^2 f \, d\mu^i_t d t \right)	\\
		\le{}& \lim_{i \to \infty} \mu^i_{-1} (f)
		+ \lim_{i \to \infty}  \int_{-1}^{0} \int  |H_i|  |\nabla f|  \, d\mu^i_t d t 	
			&& ( f \ge 0)	\\
		={}& \mu_{\tilde \C} (f) + \lim_{i \to \infty}  \int_{-1}^{0} \int  |H_i|  |\nabla f|  \, d\mu^i_t d t. 
	\end{align*}
	For any $0 < \delta \ll 1$, the integral term can be estimated as follows:
	\begin{align*}
		&\lim_{i \to \infty}  \int_{-1}^{0} \int  |H_i|  |\nabla f|  \, d\mu^i_t d t	\\
		\le{}& \lim_{i \to \infty} \|  f \|_{C^1} \int_{-1}^0 \int_{B_R} |H_i| d\mu^i_t dt	\\
		={}& \lim_{i \to \infty} \|  f \|_{C^1} \lambda_i^{n-1} \int_{-1}^0 \int_{B_{R\lambda_i^{-1}} } |H| d\mu_{t\lambda_i^{-2}} dt	\\
		={}&\lim_{i \to \infty} \|  f \|_{C^1} \lambda_i^{n+1} \int_{-\lambda_i^{-2}}^0 \int_{B_{R\lambda_i^{-1}} } |H| d\mu_{\tau } d \tau	
			&& ( \tau = t \lambda_i^{-2} )\\
		\le{}&\lim_{i \to \infty} \|  f \|_{C^1} \lambda_i^{n+1} 
		\int_{-\lambda_i^{-2}}^0 \left( \int_{B_{R\lambda_i^{-1}} } |H|^p d\mu_{\tau } \right)^{1/p} 
		\mu_\tau( B_{R \lambda_i^{-1} } )^{ \frac{p-1}{p} } d \tau 	\\
		\le{}& \| f \|_{C^1} \| H \|_{L^\infty L^p ( B_\delta \times (a, b) ) }
		\lim_{i \to \infty} \lambda_i^{n-1} 
		\left( \sup_{-\lambda_i^{-2} \le \tau \le 0 } \mu_\tau ( B_{R \lambda_i^{-1} } ) \right)^{\frac{p-1}{p}} \\
		={}& \| f \|_{C^1} \| H \|_{L^\infty L^p ( B_\delta \times (a, b) ) } R^{  \frac{ n(p-1)}{p} } 
		\lim_{i \to \infty} \lambda_i^{-1 + \frac{n}{p} } 
		\left( \sup_{-\lambda_i^{-2} \le \tau \le 0 } \frac{ \mu_\tau ( B_{R \lambda_i^{-1} } ) }{ R^n \lambda_i^{-n} }  \right)^{\frac{p-1}{p}}.
	\end{align*}
	By Lemma \ref{lem area ratio bound}, the $\frac{ \mu_\tau ( B_{R \lambda_i^{-1} } ) }{ R^n \lambda_i^{-n} }$ term is uniformly bounded above by say $C_0 < \infty$.
	Since $p >n$,
	$$\mu_\C(f) \le \mu_{\tilde \C} (f) + C_0 \| f \|_{C^1} \| H \|_{L^\infty L^p ( B_\delta \times (a, b) ) } R^{  \frac{ n(p-1)}{p} } C_0 \lim_{i \to \infty} \lambda_i^{-1 + \frac{n}{p} }
	= \mu_{\tilde \C}(f).$$
	Thus, $\mu_\C(f) \le \mu_{\tilde \C} (f)$ for all $f \in C^1_c (\R^N)$ with $f \ge 0$.
	It then follows from a limiting argument that $\mu_\C \le \mu_{\tilde \C}$. 
	
	Note that, since $\C$ is dilation invariant,
		$$\frac{ \mu_\C ( B_r( \mathbf 0 ) )}{ \omega_n r^n} = \frac{1}{ ( 4 \pi (-t) )^{n/2} } \int e^{- \frac{ |x|^2}{ 4 (-t)} } d \mu_\C 
		\qquad \forall 0 < -t, r < \infty$$
	and similarly for $\tilde \C$.
	$\mu_\C \le \mu_{\tilde \C}$ implies
		$$\frac{ \mu_\C ( B_r( \mathbf 0 ) )}{ \omega_n r^n } \le \frac{ \mu_{\tilde \C} ( B_r( \mathbf 0 ) )}{ \omega_n r^n }		\qquad \forall 0 < r < \infty.$$
		
	We claim that the reverse inequality
		$$\frac{ \mu_\C ( B_r( \mathbf 0 ) )}{ \omega_n r^n } \ge \frac{ \mu_{\tilde \C} ( B_r( \mathbf 0 ) )}{ \omega_n r^n }$$	
	also holds (for some or equivalently all $0 < r < \infty$).
	First, observe there exists a time $\tau < 0$ such that $\mu_{ \tau \lambda_i^{-2} }$ is an integral $n$-varifold with generalized mean curvature $ H_{ \tau \lambda_i^{-2}} \in L^p_{loc} (U)$ for all $i$ and
		$$\| H_{ \tau \lambda_i^{-2}} \|_{L^p(K)} \le \| H \|_{L^\infty L^p(K \times (a,b))} < \infty \qquad (\forall K \Subset U) ,$$
	where as usual $H$ denotes the mean curvature of the flow $(\mu_t)$.
	Fix $R> 0$ such that $\ol{B}_R = \ol{B}_R(\mathbf 0 ) \Subset U$.
	Let 
		$$F_i ( \rho ) \doteqdot e^{ \| H_{\tau \lambda_i^{-2} } \|_{L^p( B_R )} \frac{1}{ 1 - \frac n p} \rho^{1 - \frac np} } \qquad \text{and}
		\qquad \ol{F}(\rho) \doteqdot e^{ \| H \|_{L^\infty L^p(B_R \times( a,b))} \frac{1}{ 1 - \frac np} \rho^{1 - \frac np} }$$
	as in the monotonicity formula \eqref{simple mono forla eqn}.
	Note
		$$1 \le F_i (\rho ) \le \ol{F}( \rho )	\qquad \forall \rho > 0.$$
	It follows from the monotonicity formula \eqref{simple mono forla eqn} that, for any $0 < \rho < R$ and any $r > 0$,
	\begin{align*}
		r^{-n} \mu_{\tilde \C} ( B_r )
		&= \lim_{i \to \infty} r^{-n} \mu^i_{\tau}( B_r)	\\
		&= \lim_{i \to \infty} \frac{\mu_{ \tau \lambda_i^{-2}}( B_{r \lambda_i^{-1}} ) }{ r^n \lambda_i^{-n}}	\\
		&= \lim_{i \to \infty} \frac{1}{ F_i( r \lambda_i^{-1} ) } \left( F_i( r \lambda_i^{-1} ) \frac{\mu_{ \tau \lambda_i^{-2}}( B_{r \lambda_i^{-1}} ) }{ r^n \lambda_i^{-n}} + F_i( r \lambda_i^{-1} ) - F_i ( r \lambda_i^{-1} ) \right)  \\
		&\le \limsup_{i \to \infty} \frac{1}{ F_i( r \lambda_i^{-1} ) } 
		\left[ F_i ( \rho ) \frac{ \mu_{\tau  \lambda_i^{-2} } ( B_\rho ) }{ \rho^n} + F_i( \rho ) - F_i ( r \lambda_i^{-1} ) \right]	
		&& \eqref{simple mono forla eqn}	\\
		&\le \limsup_{i \to \infty} 
		\left[ \ol{F}(\rho ) \frac{ \mu_{\tau \lambda_i^{-2} } ( B_\rho ) }{ \rho^n} + \ol{F}( \rho ) - 1 \right]	\\
		&=  \ol{F}( \rho ) \frac{ \mu_{V_{0}} ( B_\rho ) }{ \rho^n} +  \ol{F}( \rho ) - 1 .
	\end{align*}
	Thus,
		$$\frac{\mu_{\tilde \C } ( B_r ) }{ r^n} \le \ol{F}( \rho ) \frac{ \mu_{V_{0}} ( B_\rho ) }{ \rho^n} +   \ol{F}( \rho ) - 1  .$$
	Taking $\rho \searrow 0$ reveals
		$$ \frac{ \mu_{\tilde \C} ( B_r ) }{ r^n} 
		\le \lim_{\rho \searrow 0} \ol{F}( \rho ) \frac{ \mu_{V_{0}} ( B_\rho ) }{ \rho^n} +   \ol{F}( \rho ) - 1
		= \lim_{\rho \searrow 0} \frac{ \mu_{V_{0}} ( B_\rho ) }{ \rho^n} 
		= \omega_n \theta_{V_0} ( \mathbf 0)
		= \frac{ \mu_\C ( B_r ) }{r^n}.$$
	
	In summary, we have shown that 
		$$\mu_\C \le \mu_{\tilde \C} \qquad \text{and} \qquad \frac{ \mu_\C( B_r) }{\omega_n r^n} \ge \frac{ \mu_{\tilde \C}( B_r) }{\omega_n r^n} \quad \forall 0 < r < \infty.$$
	In particular, $\mu_\C( B_r )= \mu_{\tilde \C} (B_r)$ for all $0 < r < \infty$.	
	It follows that $\C = \tilde \C$.
	Indeed, if not, then there exists a bounded subset $A \subset \R^N$ such that $\mu_\C(A) < \mu_{\tilde \C} (A)$.
	Taking $R > 0$ large enough so that $B_R \supset A$ would then imply
	$$\mu_\C (B_R) = \mu_\C ( B_R \setminus A ) + \mu_\C (A)
		< \mu_{\tilde \C} ( B_R \setminus A ) + \mu_{\tilde \C}(A)
		= \mu_{\tilde \C} (B_R) = \mu_\C (B_R)$$
	a contradiction.	
	
	Thus, $\C = \tilde \C$ and in particular
		$$\theta_{V_0} ( \mathbf 0) = \Theta_\mu ( \mathbf 0, 0).$$
	This completes the proof.
\end{proof}

As an immediate consequence of Theorem \ref{thm blow-ups}, we deduce that the uniqueness of tangent \emph{cones} of time-slices is equivalent to the uniqueness of tangent \emph{flows}. 

\begin{cor} \label{cor uniqueness equivalence}
	Let $(\mu_t)_{t \in (a, b)}$ be an integral $n$-Brakke flow in $U \subset \R^N$ with locally uniformly bounded areas and generalized mean curvature $H  \in L^\infty L^p_{loc}(U \times (a,b))$ for some $p \in (n, \infty]$.
	Let $(V_t)_{t \in (a, b] }$ be the family of varifolds as in Lemma \ref{lem nice representative}.
	
	Let $(x_0, t_0) \in U \times ( a, b]$.
	Then
	\begin{multline*}
		\{ \C : \C \text{ is a tangent cone of $V_{t_0}$ at $x_0$} \} \\
		= \{  \C : \mu_{ \C} \text{ is a tangent flow of $(\mu_t)$ at $(x_0, t_0)$} \}.
	\end{multline*}

	In particular, the tangent cone $V_{t_0}$ at $x_0$ is unique if and only if the tangent flow of $(\mu_t)$ at $(x_0, t_0)$ is unique.
	In other words,
		$$( \eta_{x_0, \lambda_i})_\sharp V_{t_0} \rightharpoonup \C$$
	for every sequence $\lambda_i \nearrow +\infty$ if and only if
		$$( \cl{D}_{x_0, t_0; \lambda_i} \mu )_t \rightharpoonup \mu_\C \qquad ( \forall t < 0)$$
	for every sequence $\lambda_i \nearrow +\infty$.
\end{cor}

\begin{remark}
	Let $(\mu_t)_{t \in (a,b)}$ be an integral Brakke flow with mean curvature bounds as in Theorem \ref{thm blow-ups}, and consider the associated flow of varifolds $(V_t)_{t \in (a,b]}$ as in Lemma \ref{lem nice representative}.
	By combining Theorem \ref{thm blow-ups} with White's local regularity theorem \cite{White05} (and its generalization to integral $n$-Brakke flows \cite{ST22}),
	it follows that
	if $t_0 \in (a,b]$ and $V_{t_0}$ also has \emph{unit density} (i.e. $\theta_{V_{t_0}}(x) = 1$ for $\mu_{V_{t_0}}$-a.e. $x \in U$),
	then the singular set of the flow $\cl{M} = ( V_t)_{t \in (a,t_0]}$ in the $t=t_0$ time-slice has $n$-dimensional Hausdorff measure 0, that is $\cl{H}^n(\text{sing}_{t_0} \cl{M} ) = 0$.
	This recovers Brakke's main regularity theorem \cite[6.12]{Brakke78}
	in the special case of flows with mean curvature bounds (see also \cite[Theorem 5.3]{Ecker04} and \cite[Theorem 3.2]{KT14}).
	We refer the interested reader to \cite{Brakke78, Ecker04, KT14} for precise definitions.
\end{remark}

\section{Uniqueness of Tangent Flows Given By Regular Cones} \label{sect Uniqueness of Tangent Flows}

\begin{definition}
	Define the \emph{link} $L(\C)$ of a dilation invariant set $\C \subset \R^N$ to be
		$$L( \C ) \doteqdot \C \cap \Sph^{N-1}.$$
	$\C$ is said to be a \emph{regular cone} if $L(\C)$ is a smooth, (properly) embedded submanifold of $\Sph^{N-1}$.
\end{definition}

If $\C \subset \R^N$ is a regular cone, then $\C \setminus \{ \mathbf 0 \}$ is a smooth submanifold of $\R^N$ and we write $\C^n$ when $\C \setminus \{ \mathbf 0 \}$ has dimension $n$.
Note that a regular cone $\C^n \subset \R^N$ naturally gives a dilation invariant integral $n$-varifold (of multiplicity one) with associated measure $\cl{H}^n \mres \C$ on $\R^N$.

The main result of this section is the following uniqueness result which restates Theorem \ref{meta thm 2}:

\begin{thm} \label{thm uniqueness regular cones}
	Let $(\mu_t)_{t \in (a,b)}$ be an integral $n$-Brakke flow in $U \subset \R^N$ with locally uniformly bounded areas and generalized mean curvature $H \in L^\infty L^\infty_{loc}(U \times (a,b))$.
	If $(\mu_t)$ has a tangent flow $\mu_\C$ at $(x_0,t_0) \in U \times (a, b]$ given by a regular cone $\C^n$ (with multiplicity one),
	then $\mu_\C$ is the unique tangent flow of $(\mu_t)$ at $(x_0, t_0)$.
\end{thm}

The proof essentially follows from \cite[\S 7, Theorem 5]{Simon83} applied to the time slice $V_{t_0}$.
However, \cite[\S 7, Theorem 5]{Simon83} requires a regularity assumption \cite[equation (7.23)]{Simon83} which we must verify for $V_{t_0}$.
Informally, this regularity assumption states that if $V_{t_0}$ is a small $C^{1, \alpha}$-graph over the regular cone $\C$ in an annulus, then the mean curvature $H_{V_{t_0}}$ of $V_{t_0}$ has interior $C^2$-bounds.

Since $V_{t_0}$ comes from a Brakke flow, we can prove $V_{t_0}$ satisfies this regularity assumption \cite[equation (7.23)]{Simon83} as follows:
\begin{enumerate}
	\item show that $C^{1, \alpha}$-graphicality propagates outward in space and backward in time (Lemma \ref{lem graphicality propogates}), and
	
	\item apply interior estimates to improve the $C^{1, \alpha}$ bounds to $C^\infty$ estimates for $V_t$ and apply interior estimates to the evolution equation for $H = H_{V_t}$ to obtain $\| H \|_{C^2} \lesssim \| H \|_{C^0}$ (Lemma \ref{lem interior H ests}).
\end{enumerate}
The remainder of this section rigorously carries out this argument to prove Theorem \ref{thm uniqueness regular cones}.

\begin{definition}
	In what follows we use $A_{r, R}( x_0)$ to denote the open annulus $A_{r, R}(x_0) = B_R( x_0) \setminus \ol{ B}_r (x_0) $ and $A_{r, R} = A_{r, R} (\mathbf 0)$.
	
	For $\C^n \subset \R^N$ a regular cone, we slightly abuse notation and write $u : \C \cap A_{r, R} \to T^\perp \C$ to mean a function
	 $u : \C \cap A_{r, R} \to \R^N$ such that $u(x) \in T^\perp_x \C$ for all $x \in \C \cap A_{r,R}$.
	 For $u : \C \cap A_{r,R} \to T^\perp \C$, denote
		$$G_{r, R} (u) \doteqdot \left \{  \frac{ x + u(x) }{ \sqrt{ 1 + \frac{ |u(x)|^2 }{ |x|^2} } } : x \in \C \cap A_{r, R} \right\} \subset \R^N.$$
\end{definition}

Note that if $\C$ is a regular cone and $u : \C \cap A_{r, R} \to T^\perp \C$ is a $C^2$-function with $\frac{ u(x)}{|x|}$ sufficiently small in $C^1$, then $G_{r, R}(u)$ is a properly embedded $C^2$-submanifold.

\begin{definition}
	For $\C$ a regular cone, let $\| \cdot \|_{C^k_*}$ and $\| \cdot \|_{C^{k, \alpha}_*}$ ($0 < \alpha \le 1$) denote the standard $C^k$ and $C^{k, \alpha}$ norms respectively.
	For example, if $u : \C \cap \Omega \to \R^N$ then
		$$\| u \|_{C^{k, \alpha}_*( \C \cap \Omega) } = \sum_{j = 0}^k \sup_{x \in \C \cap \Omega} | \nabla^j_\C u |(x) + \sup_{x \ne y \in \C \cap \Omega} \frac{ | \nabla^k_\C u (x) - \nabla^k_\C u(y) |}{  |x - y|^\alpha}.$$
		
	For $0 < r < R < \infty$ and $u : \C \cap A_{r, R} \to \R^N$, define scale-invariant $C^k$ and $C^{k, \alpha}$ norms by
	\begin{align*}
		\| u \|_{C^k ( \C \cap A_{r, R} ) } &\doteqdot \sum_{j = 0}^k \sup_{x \in \C \cap A_{r, R} } |x|^{j - 1} \sup_{y \in B_{\frac{|x|}2}(x) \cap \C \cap A_{r, R}   }| \nabla^j_\C u |(y) \\
		&= \sum_{j = 0}^k \sup_{x \in \C \cap A_{r, R} } |x|^{j - 1} \| \nabla^j_\C u \|_{C^0 ( B_{\frac{|x|}2}(x) \cap \C \cap A_{r, R})} 
	\end{align*}
	\begin{align*}
		& \| u \|_{C^{k, \alpha}( \C \cap A_{r, R} ) } \\
		\doteqdot{}& 
			\| u \|_{C^k( \C \cap A_{r, R} ) } 
			+ \sup_{x \in \C \cap A_{r, R} } |x|^{k -1 + \alpha}
			\sup_{y \ne z \in B_{\frac{|x|}2} (x) \cap \C \cap A_{r, R} }
			\frac{ | \nabla^k_\C u (y) - \nabla^k_\C u(z) | }{ | y - z |^\alpha} \\
		={}& \| u \|_{C^k( \C \cap A_{r, R} ) } 
			+ \sup_{x \in \C \cap A_{r, R} } |x|^{k -1 + \alpha} [ \nabla^k u ]_{C^{ \alpha}_* ( B_{\frac{|x|}2} (x) \cap \C \cap A_{r, R}  ) }
	\end{align*}
	where $[ \, \cdot \, ]_{C^{ \alpha}_*}$ denotes the standard $C^{ \alpha}$ semi-norm.
	
	For functions $u : \C \cap A_{r, R} \times (a, b) \to \R^N$ that also depend on time $t \in (a,b)$, 
	denote backward parabolic neighborhoods as
		$$P_r (x, t) \doteqdot B_r(x) \times (t - r^2, t)$$
	and define scale-invariant $C^k$ and $C^{k, \alpha}$ norms by	
	\begin{align*}
		& \| u \|_{C^{k} ( \C \cap A_{r, R} \times (a,b) ) } \\
		\doteqdot{}& \sum_{2i + j \le k} \sup_{(x,t) \in \C \cap A_{r,R} \times (a,b)} |x|^{2i + j - 1} \| \partial_t^i \nabla_\C^j u \|_{C^0_* ( P_{\frac{|x|}2 } (x,t) \cap (\C \cap A_{r,R} \times (a,b) ) ) } \\
		&+ \sum_{0 < \frac{k - 2i - j} 2 < 1} \sup_{(x,t) \in \C \cap A_{r,R} \times (a,b)} |x|^{k-1+\alpha} 
			[\partial_t^i \nabla_\C^j u ]_{t, C^{\frac{k-2i-j}2}_* ( P_{\frac{|x|}2 } (x,t) \cap (\C \cap A_{r,R} \times (a,b) ) ) },
	\end{align*}
	
	\begin{align*}
		& \| u \|_{C^{k, \alpha} ( \C \cap A_{r, R} \times (a,b) ) } \\
		\doteqdot{}& \sum_{2i + j \le k} \sup_{(x,t) \in \C \cap A_{r,R} \times (a,b)} |x|^{2i + j - 1} \| \partial_t^i \nabla_\C^j u \|_{C^0_* ( P_{\frac{|x|}2 } (x,t) \cap (\C \cap A_{r,R} \times (a,b) )) } \\
		&+ \sum_{2i + j = k} \sup_{(x,t) \in \C \cap A_{r,R} \times (a,b)} |x|^{k - 1 + \alpha} [ \partial_t^i \nabla_\C^j u ]_{x,C^\alpha_{*}( P_{\frac{|x|}2 } (x,t) \cap (\C \cap A_{r,R} \times (a,b) ) )} \\
		&+ \sum_{0 < \frac{k + \alpha - 2i - j} 2 < 1} \sup_{(x,t) \in \C \cap A_{r,R} \times (a,b)} |x|^{k-1+\alpha} 
			[ \partial_t^i \nabla_\C^j u ]_{t,C^{\frac{k + \alpha - 2i - j} 2}_{*}( P_{\frac{|x|}2 } (x,t) \cap (\C \cap A_{r,R} \times (a,b) ) )} .
	\end{align*}
	Here, $[ \, \cdot \, ]_{x, C^\alpha_*} $ and $[ \, \cdot \, ]_{t, C^\alpha_*}$ denote the standard $C^\alpha$ semi-norms in the variables $x$ and $t$ respectively.
	Namely,
	\begin{align*}
		[u]_{x, C^\alpha_*(\C \times (a,b) \cap \Omega)} 
		&\doteqdot \sup_{ (x,t) \ne (x', t) \in \C \times (a,b) \cap \Omega } 
		\frac{| u(x,t) - u(x',t)|}{|x - x'|^\alpha}, \text{ and}
		\\
		[u]_{t, C^\alpha_*(\C \times (a,b) \cap \Omega)} 
		&\doteqdot \sup_{ (x,t) \ne (x, t') \in \C \times (a,b) \cap \Omega } 
		\frac{| u(x,t) - u(x,t')|}{|t - t'|^\alpha}.
	\end{align*}
\end{definition}

Since we use the $C^{1, \alpha}$ norm most often, we note explicitly that 
	\begin{align*}
		& \| u \|_{C^{1, \alpha}( \C \cap A_{r,R} \times (a,b)) } \\
		={}& \sup_{x,t} |x|^{-1} \sup_{(y,s) \in P_{\frac{|x|}2} (x,t)} |u(y,s) | 
		+ \sup_{x,t}  \sup_{(y,s) \in P_{\frac{|x|}2} (x,t)} |\nabla_\C u(y,s) |	 \\
		&+ \sup_{x,t} |x|^{\alpha} \sup_{(y,s)\ne (y',s) \in P_{\frac{|x|}2} (x,t)} \frac{|\nabla_\C u(y,s) - \nabla_\C u(y',s)|}{|y - y'|^\alpha} \\
		&+ \sup_{x,t} |x|^{\alpha} \sup_{(y,s)\ne (y,s') \in P_{\frac{|x|}2} (x,t)} \frac{| u(y,s) -  u(y,s')|}{|s - s'|^{ \frac{1+\alpha} 2} }
	\end{align*}
	where also the suprema above are restricted to points in $\C \cap A_{r,R} \times (a,b)$.

\begin{remark}
	While the $C^{k, \alpha}$ norms defined above are somewhat non-standard,
	they have been chosen so that they satisfy the following properties, the proofs which have been left as exercises to the reader.
	\begin{enumerate}
		\item (Parabolic Scaling Invariance) If $u : \C \cap A_{r,R} \times (a,b) \to \R^N$, $\lambda > 0$, and $\tilde u : \C \cap A_{\lambda r, \lambda R} \times ( \lambda^2 a, \lambda^2 b) \to \R^N$ is given by $\tilde u (x,t) = \lambda u ( x/ \lambda, t / \lambda^2 )$ then
			$$\| \tilde u \|_{C^{k, \alpha} ( \C \cap A_{\lambda r, \lambda R} \times (\lambda^2 a, \lambda^2 b) ) } = \| u \|_{C^{k , \alpha} ( \C \cap A_{r, R} \times ( a, b) )}.$$

		\item (Time Translation Invariance) If $u : \C \cap A_{r,R} \times (a,b) \to \R^N$, $t_0 \in \R $, and $\tilde u : \C \cap A_{ r,  R} \times (  a + t_0, b + t_0) \to \R^N$ is given by $\tilde u (x,t) =  u ( x, t - t_0 )$ then
			$$\| \tilde u \|_{C^{k, \alpha} ( \C \cap A_{ r,  R} \times ( a+t_0,  b+t_0) ) } = \| u \|_{C^{k , \alpha} ( \C \cap A_{r, R} \times ( a, b) )}.$$

		\item (Equivalent to Standard H{\"o}lder Norms) There exists $C = C(r,R, k, \alpha)$ such that  
			$$ C^{-1} \| u \|_{C^{k, \alpha}_* ( \C \cap A_{r, R} \times ( a, b) )} 
			\le \| u \|_{C^{k, \alpha} ( \C \cap A_{r, R} \times ( a, b) )} 
			\le C \| u \|_{C^{k, \alpha}_* ( \C \cap A_{r, R} \times ( a, b) )} .$$			
	\end{enumerate}
	Similar properties hold for the space-time $C^k$ norms and the spatial $C^k$ and $C^{k, \alpha}$ norms.
\end{remark}

\begin{remark}
	The choice of using radius $\frac{|x|}2$ balls in the above definitions was somewhat arbitrary.
	Indeed, it can be shown through a covering argument that, for any $L > 1$, replacing ``$\frac{|x|}2$" with ``$\frac{|x|}L$" in the above definitions gives an equivalent norm, that is, the norms differ by a factor of $C = C(k, \alpha , L)$.
\end{remark}

\begin{lem}[Graphicality propagates out and back] \label{lem graphicality propogates}
	Let $( \mu_t)_{t \in (a,b)}$ be an integral $n$-Brakke flow in $U \subset \R^N$ with locally uniformly bounded areas and generalized mean curvature $H \in L^\infty L^p_{loc}(U \times (a,b))$ for some $p \in (n, \infty]$.
	Let $(V_t)_{t \in (a, b]}$ be the  associated integral $n$-Brakke flow from Lemma \ref{lem nice representative}.
	
	Let $(x_0, t_0) \in U \times ( a, b]$ and let $\C^n \subset \R^N$ be a minimal regular cone.
	For any $\epsilon > 0$,
	there exists $r_0, \delta > 0$ such that the following holds for all $0 < \rho < r_0$:
	
	if 
	$$(V_{t_0} -x_0) \cap A_{\rho/2, \rho} = G_{\rho/2, \rho}(u)$$ 
	for some $u : \C \cap A_{\rho/2, \rho} \to T^\perp C$ with 
	$$\| u \|_{C^{1, \alpha} ( \C \cap A_{\rho/2, \rho} ) } \le \delta,$$
	then 
		$$(V_t - x_0) \cap A_{\rho/4, 2\rho} = G_{\rho/4, 2\rho} ( \tilde u ( \cdot , t) )
		\qquad ( \forall t \in [t_0 - 4\rho^2, t_0] )$$
	for some extension $\tilde u : (\C \cap A_{\rho/4, 2\rho}) \times [t_0 - 4 \rho^2, t_0] \to T^\perp \C$ of $u$ with 
	$$\| \tilde u \|_{C^{1, \alpha} ( \C \cap A_{\rho/4, 2\rho} \times [t_0 - 4 \rho^2, t_0 ] )} \le \epsilon .$$
\end{lem}
\begin{proof}
	By translation, assume without loss of generality that $(x_0, t_0) = (\mathbf 0 , 0)$.
	Suppose the lemma were false for the sake of contradiction.
	Then we can take a sequence $r_i = \delta_i \searrow 0$ and obtain $\rho_i \in ( 0 , r_i )$ where the implication fails.
	That is,
	$V_0 \cap A_{\rho_i/2, \rho_i} = G_{\rho_i/2, \rho_i} (u_i)$ is a $C^{1, \alpha}$-graph over $\C \cap A_{\rho_i/2, \rho_i}$ with  $\| u_i \|_{C^{1, \alpha}(\C \cap A_{\rho_i/2, \rho_i} )} \le \delta_i$,
	but 
	$(V_t)$ is not a $C^{1, \alpha}$-graph over $\C \cap A_{\rho_i/4, 2 \rho_i} \times [t_0 - 4 \rho_i^2, t_0]$ with $C^{1, \alpha}$-norm bounded by $\epsilon$ in this region.
	
	Parabolically dilate $V_t$ by $\lambda_i \doteqdot \frac{1}{\rho_i} \to +\infty$
	to obtain $V^i_t \doteqdot (\eta_{\lambda_i})_\sharp V_{t \lambda_i^{-2}}$ and set $\mu^i_t \doteqdot \mu_{V^i_t}$.
	After passing to a subsequence, Theorem \ref{thm blow-ups} applied to the Brakke flow $(\mu_{V_t})$ implies there exists a stationary, dilation invariant varifold $\C'$ such that 
		$$V^i_0 \rightharpoonup \C'
		\qquad \text{ and } \qquad
		\mu^i_t  \rightharpoonup \mu_{\C'} 
		\quad ( \forall t < 0)$$
	as $i \to \infty$.
	Since $\C'$ is dilation invariant and 
	\begin{gather*}
	V_0^i \cap A_{1/2, 1} = G_{1/2, 1} \left( \lambda_i u_i ( \cdot/ \lambda_i ) \right) \\
	\text{with } \| \lambda_i u_i ( \cdot/ \lambda_i) \|_{C^{1, \alpha} ( \C \cap A_{1/2, 1} )}
	= \| u_i \|_{C^{1, \alpha} ( \C \cap A_{\rho_i/2, \rho_i} )} \le \delta_i \to 0,
	\end{gather*}
	it follows that in fact $\C' = \C$.
		
	Now the stationary flow $( \mu^\infty_t = \mu_C )_{t \le 0}$ given by $\C$ has Gaussian density 
		$$\Theta_{\mu^\infty} (x, t) = 1	\qquad  \forall (x, t) \in \ol{A_{1/6, 4}} \times [ -6, 0]$$ 
	since $\C$ has smooth link.
	The upper semi-continuity of Gaussian density then implies that for any $\sigma > 0$
		$$ \Theta_{\mu^i}(x,t) < 1 + \sigma		\qquad \forall (x, t) \in \ol{A_{1/6, 4}} \times [ -6, 0]$$
	for all $i \gg 1$ sufficiently large.
	
	By White's local regularity theorem \cite{White05} (and its generalization to integral Brakke flows \cite{ST22}),
	it follows that, for $i \gg 1$, 
	$V^i_t$ is a smooth mean curvature flow in $A_{1/6,4} \times [-6,0]$ with second fundamental form bounded by a dimensional constant $C = C(N) < \infty$ in $A_{1/5, 3} \times [-5, 0]$.
	Interior regularity for mean curvature flow then implies that the convergence $V^i_t \xrightarrow[i \to \infty]{} \C$ is smooth on $A_{1/4, 2} \times [ -4, 0]$.
	It follows that, for all $i \gg 1$,
	$$V^i_t \cap A_{1/4, 2} = G_{1/4,2} ( \tilde w_i ( \cdot, t) )		\qquad ( \forall t \in [-4,0])$$
	for some $\tilde w_i : \C \cap A_{1/4, 2} \times [-4,0] \to T^\perp C$ extending $\lambda_i u_i ( \cdot / \lambda_i )$ with
	$$\| \tilde w_i \|_{C^{3, \alpha} ( C \cap A_{1/4, 2} \times [-4, 0] ) } \xrightarrow[i \to \infty]{} 0.$$
	Since 
		$$\| \tilde w_i ( \cdot, 0) \|_{C^{1, \alpha} ( \C \cap A_{1/2, 1} ) } 
		= \| \lambda_i \tilde u_i ( \cdot / \lambda_i ) \|_{C^{1, \alpha} ( \C \cap A_{1/2, 1} )} \le \delta_i \to 0,$$
	and derivatives of $\tilde w_i$ converge to 0 on $A_{1/4, 2} \times [-4, 0]$, we have that in fact
		$$\| \tilde w_i \|_{C^{1, \alpha} ( \C \cap A_{1/4, 2} \times [ -4, 0] )} \le \epsilon$$
	for all $i \gg 1$.
	Undoing dilations gives that for $i \gg 1$
		$$V_t \cap A_{\rho_i/4, 2\rho_i} = G_{\rho_i/4, 2\rho_i} ( \tilde u_i(\cdot, t) )		\qquad (\forall  t \in [-4\rho_i^2, 0])$$
	for $\tilde u_i(x,t) = \frac{1}{\lambda_i} \tilde w_i ( x \lambda_i, t \lambda_i^2 ) : \C \cap A_{\rho_i/4, 2\rho_i} \to T^\perp \C$
	extending $u_i$ and
		$$\| \tilde u_i \|_{C^{1, \alpha} ( \C \cap A_{\rho_i/4, 2\rho_i}  \times [- 4\rho_i^2, 0] ) } \le \epsilon,$$
	which contradicts the choice of the $r_i, \delta_i, \rho_i$.
\end{proof}

\begin{lem} \label{lem interior H ests}
	Let $(\mu_t)_{t \in (a,b)}$ be an integral $n$-Brakke flow in $U \subset \R^N$ with locally uniformly bounded areas and generalized mean curvature $H \in L^\infty L^\infty_{loc} ( U \times (a, b))$.
	Let $(V_t)_{t \in (a, b]}$ be the associated family of varifolds as in Lemma \ref{lem nice representative}.
	Let $(x_0, t_0 ) \in U \times (a, b]$ and 
	let $\mathbf C$ be a regular cone.
	
	There exists $\beta, C, r_0 > 0$ such that for all $0 < \rho < r_0$ the following holds:
	
	if 
		$$(V_{t_0} - x_0) \cap A_{\rho/2,\rho} = G_{\rho/2, \rho} ( u ) $$
	for some $u : \mathbf C \cap A_{\rho/2, \rho} \to T^\perp \mathbf C$ with 
		$$\| u \|_{C^{1, \alpha} (\mathbf C \cap A_{\rho/2, \rho} )} \le \beta,$$
	then for any $0 < \sigma < \rho$
	\begin{multline*}
		| \nabla_{V_{t_0}} H_{V_{t_0}} | 
		\le \frac{ C \| H \|_{L^\infty L^\infty(A_{\rho/2, \rho} (x_0) \times (t_0 - \rho^2 , t_0))} }{ \sigma} \\
		\le \frac{ C \| H \|_{L^\infty L^\infty(A_{r_0/2, r_0} (x_0) \times (t_0 - r_0^2 , t_0))} }{ \sigma} <\infty
		\quad \text{and}
	\end{multline*}
	\begin{multline*}
		| \nabla^2_{V_{t_0}} H_{V_{t_0}} | 
		\le \frac{ C \| H \|_{L^\infty L^\infty(A_{\rho/2, \rho}(x_0) \times (t_0 - \rho^2,t_0)) } }{ \sigma^2} \\
		\le \frac{ C \| H \|_{L^\infty L^\infty(A_{r_0/2, r_0}(x_0) \times (t_0 - r_0^2,t_0)) } }{ \sigma^2} < \infty
	\end{multline*}
	on $ V_{t_0}  \cap A_{ \rho/2 + \sigma, \rho - \sigma }(x_0)$.
\end{lem}
\begin{proof}
	Throughout, we assume $0 < r_0 \ll 1$ is small enough so that $\ol{B}_{2 r_0} (x_0 ) \Subset U$ and $(t_0 - 4 r_0^2, t_0 ) \subset (a,b)$.
	By translation, assume without loss of generality that $(x_0, t_0 )= ( \mathbf 0, 0)$.
	Assume $0 < \rho < r_0$ and 
		$$V_{0}  \cap A_{\rho/2,\rho} = G_{\rho/2, \rho} ( u ) $$
	for some $u : \mathbf C \cap A_{\rho/2, \rho} \to T^\perp \mathbf C$ with 
		$$\| u \|_{C^{1, \alpha} (\mathbf C \cap A_{\rho/2, \rho} )} \le \beta.$$
	
	By Lemma \ref{lem graphicality propogates}, for any $\epsilon > 0$, we can assume $\beta, r_0 \ll 1$ are sufficiently small (depending on $\epsilon$) so that 
		$$V_{t}  \cap A_{\rho/4,2\rho} = G_{\rho/4, 2\rho} ( \tilde u( \cdot, t) )		\qquad \forall t \in [-4\rho^2, 0] $$
	for some extension $\tilde u : \mathbf C \cap A_{\rho/4, 2\rho} \times[-4 \rho^2, 0] \to T^\perp \mathbf C$ of $u$ with 
		$$\| \tilde u \|_{C^{1, \alpha} (\mathbf C \cap A_{\rho/4, 2\rho} \times [-4\rho^2, 0] )} \le \epsilon.$$
	Consider the parabolically rescaled flow $W_t \doteqdot ( \eta_{1/\rho} )_\sharp V_{t \rho^2}$ and note
		$$W_t \cap A_{1/4, 2} = G_{1/4, 2} ( \tilde w ( \cdot, t ) )	\qquad \forall t \in [-4, 0]$$
	where
		$$\tilde w(x, t) \doteqdot \frac{1}{\rho} \tilde u ( x \rho, t \rho^2 ),		\qquad
		\| \tilde w \|_{C^{1, \alpha} ( A_{1/4, 2} \times [ -4, 0] )}
		= \| \tilde u \|_{C^{1, \alpha} ( A_{\rho/4, 2 \rho} \times [ -4 \rho^2, 0] )} \le \epsilon.$$
	If $\epsilon = \epsilon ( N, \mathbf C)$ is sufficiently small (depending only on $N$ and $\mathbf C$),
	interior estimates (see e.g. \cite{ST22} or \cite{Tonegawa14})
	imply that $W_t$ is a smooth mean curvature flow on $A_{1/3, 3/2 } \times [-3, 0]$
	with derivative bounds on the second fundamental form $A = A_{W_t}$ of the form
	\begin{equation} \label{A Ests}
		\sup_{(x,t) \in A_{1/3, 3/2} \times [ -3,0] } | \nabla^k_{W_t} A_{W_t} | \le C_k = C_k ( N, \mathbf C)
		\qquad ( \forall k \in \N).
	\end{equation}
	
	In this region $A_{1/3, 3/2} \times [-3, 0]$ where $W_t$ is a smooth mean curvature flow, the mean curvature $H = H_{W_t}$ satisfies an evolution equation of the form
	\begin{equation} \label{H Evol Eqn}
		\partial_t H = \Delta H + \nabla A * H + A * \nabla H + A * A * H
	\end{equation}
	(see \cite[Corollary 3.8]{Smoczyk12}). 
	The bounds \eqref{A Ests} imply \eqref{H Evol Eqn} is a linear parabolic PDE system for $H_{W_t}$ in the domain $A_{1/3,3/2} \times [ -3, 0]$ with uniform $C^k$-bounds on the coefficients that depend only on $N, \mathbf C,$ and $k$.
	Interior estimates for parabolic systems (see e.g. \cite{LSU88})
	therefore imply that for some $C = C(N, \mathbf C)$
	\begin{equation} \label{interior H ests rescaled}
		\sup_{A_{1/2 + \sigma, 1 - \sigma}  } | \nabla^2_{W_0} H_{W_0} |
		\le \frac{C}{\sigma^2} \sup_{A_{1/2, 1} \times [-1, 0] }  |H_{W_t} |
		\qquad ( \forall 0 < \sigma < 1 ) .
	\end{equation}
	In terms of $V_t$, \eqref{interior H ests rescaled} becomes
	\begin{align*}
		\sup_{A_{\rho/2 + \sigma, \rho - \sigma}  } | \nabla^2_{V_0} H_{V_0} |
		={}& \frac{1}{\rho^3} \sup_{A_{1/2 + \sigma/\rho, 1 - \sigma/\rho} } | \nabla^2_{W_0} H_{W_0} |	\\
		\le{}& \frac{ C  }{(\sigma/\rho)^2 \rho^3 } \sup_{A_{1/2, 1} \times [-1, 0] }  |H_{W_t} | 
			&& \eqref{interior H ests rescaled}\\
		={}& \frac{ C  }{\sigma^2 \rho } \sup_{A_{1/2, 1} \times [-1, 0] }  |H_{W_t} | \\
		={}& \frac{ C  }{\sigma^2  } \sup_{A_{\rho/2, \rho} \times [-\rho^2, 0] }  |H_{V_t} | \\
		\le{}& \frac{ C  }{\sigma^2  } \| H \|_{L^\infty L^\infty (A_{\rho/2, \rho} \times (-\rho^2, 0)) } 
			&& (\text{Lemma } \ref{lem nice representative})
	\end{align*} 
	for all $0 < \sigma < \rho$.
	Note that in the last line $ H $ denotes the mean curvature of the Brakke flow $(\mu_t)_{t \in (a,b)}$.
	An analogous argument applies to estimate $| \nabla_{V_0} H_{V_0} |$.
\end{proof}

We can now prove Theorem \ref{thm uniqueness regular cones} by adapting the argument \cite[\S 7]{Simon83} used for tangent cones of stationary varifolds.

\begin{proof}[Proof of Theorem \ref{thm uniqueness regular cones}.]
	Throughout, we use $(V_t)_{t \in (a,b]}$ to denote the associated family of varifolds given by Lemma \ref{lem nice representative}.
	By translation, assume without loss of generality that $(x_0, t_0) = (\mathbf 0, 0)$.
	By Corollary \ref{cor uniqueness equivalence}, it suffices to show $\mathbf C$ is the unique tangent cone of $V_0$ at $\mathbf 0$.

	By Theorem \ref{thm blow-ups}, there exists some sequence $\lambda_k \nearrow +\infty$ such that
		$$V_k \doteqdot ( \eta_{\lambda_k} )_\sharp V_0 \rightharpoonup \mathbf C.$$
	and $\theta_{V_k} ( \mathbf 0 ) = \frac{ \mu_{\mathbf C} (B_1 ) }{\omega_n}$ for all $k$.	
	Note that the mean curvature $H_{V_k}$ of $V_k$ is bounded by
		$$\| H_{V_k } \|_{L^\infty ( B_2 ) } = \frac{1}{\lambda_k } \| H_{V_0} \|_{L^\infty ( B_{2/ \lambda_k} ) }
		\le \frac{1}{\lambda_k} \| H_{V_t} \|_{L^\infty L^\infty ( B_\delta \times (-\delta, 0]) } \xrightarrow[k \to \infty]{} 0$$
	so long as $\ol{B_\delta} \times (-\delta, 0) \subset U \times (a,b)$ and $k$ is large enough to ensure $2/\lambda_k < \delta$.
	Since $\mathbf C$ is smooth away from $\mathbf 0$,
	Allard's regularity theorem \cite{Allard72} (see also \cite[Ch. 5]{Simon83}) 
	implies that $V_k \cap A_{1/2,1}$ is smooth for all $k \gg 1$ sufficiently large, and the convergence $V_k \cap A_{1/2,1} \xrightarrow[k \to \infty]{} \mathbf C$ is smooth.
	In particular, for $k \gg 1$,
		$$V_k \cap A_{1/2, 1} = G_{1/2,1} ( u_k )$$
	for some $u_k : A_{1/2, 1} \to T^\perp \mathbf C$ with
		$$\| u_k \|_{C^{1, \alpha} ( A_{1/2,1}  ) } \xrightarrow[k \to \infty]{} 0.$$
	Lemma \ref{lem interior H ests} implies $V_k$ satisfies property \cite[(7.23)]{Simon83} 
	and therefore \cite[\S 7, Theorem 5]{Simon83} applies.
	That is, for $k \gg 1$,
		$$V_k \cap B_1 \setminus \{ \mathbf 0 \} = G_{0, 1} ( \tilde u_k )$$
	for some extension $\tilde u_k \in C^2 ( \mathbf C \cap B_1 \setminus \{ \mathbf 0 \} )$ of $u_k$ that satisfies
		$$\lim_{\rho \searrow 0} \frac{ \tilde u_k ( \rho \omega) }{ \rho } = \zeta_k ( \omega ) \qquad ( \omega \in L( \mathbf C ) )$$
	where $\zeta_k \in C^2 ( L ( \mathbf C ) )$ and where the convergence is in the $C^2 ( L ( \mathbf C ) )$ norm.
	Since the $V_k$ are all dilations of $V_0$, it follows that $\tilde u_k(x) = \frac{ \lambda_k}{\lambda_l} \tilde u_l \left( \frac{ \lambda_l}{ \lambda_k} x \right)$ for all $k,l \gg 1$, and thus $\zeta_k \equiv 0$ for all $k \gg 1$.
	Undoing the dilations reveals that, for a fixed $k$ suitably large,
	\begin{gather} \label{eqn global graphicality of time-slice}
		V_0 \cap B_{1/ \lambda_k} \setminus \{ \mathbf 0 \} = G_{0, 1/\lambda_k} \left( \frac{1}{\lambda_k} \tilde u_k ( \lambda_k x ) \right) \quad \text{and}	\quad
		\frac{ \frac{ 1}{\lambda_k} \tilde u_k ( \lambda_k \rho \omega) }{ \rho} 
		\xrightarrow[\rho \searrow 0]{C^2( L (\mathbf C ) )} 0.
	\end{gather}
	The uniqueness of the tangent cone now follows.
\end{proof}

As a corollary, we note that Theorem \ref{thm uniqueness regular cones} implies the flow may be written as a graph over the cone in certain space-time regions near the singularity.

\begin{cor}
	Let $(\mu_t)_{t \in (a,b)}$ be an integral $n$-Brakke flow in $U \subset \R^N$ with locally uniformly bounded areas and generalized mean curvature $H \in L^\infty L^\infty_{loc}(U \times (a,b))$.
	Let $(V_t)_{t \in (a,b]}$ be the associated integral $n$-Brakke flow from Lemma \ref{lem nice representative}.
	Assume $(\mu_t)$ has a tangent flow $\mu_\C$ at $(x_0,t_0) \in U \times (a, b]$ given by a regular cone $\C^n$.

	For any $\epsilon, C > 0$, there exists $r > 0$ such that
		$$(V_t - x_0) \cap A_{ \frac {\sqrt{t_0 - t}}C , r } = G_{ \frac {\sqrt{t_0 - t}}C , r } (u( \cdot , t)) 
		\qquad \forall t \in [ t_0 - r^2 , t_0 ]$$
	for some $u : \Omega \doteqdot \left \{ (x, t ) \in \C \times [ t_0 - r^2, t_0 ] \, | \, x \in A_{ \frac {\sqrt{t_0 - t}}C , r } \right \} \to T^\perp \C$ with $\| u \|_{C^{1, \alpha} (\Omega)  } \le \epsilon$.
\end{cor}
\begin{proof}
	For any $\delta > 0$, the proof of Theorem \ref{thm uniqueness regular cones}, namely \eqref{eqn global graphicality of time-slice}, implies there exists $r > 0$ such that $(V_{t_0} - x_0) \cap B_{r}  \setminus \{ \mathbf 0 \}$ can be written as a graph over $\C$ with $C^{1, \alpha}$-norm bounded by $\delta$.
	Let $\epsilon, C > 0$ be given.
	If $\delta = \delta( \epsilon , C ) \ll 1$ is sufficiently small and $r$ is possibly made smaller, then it follows from Lemma \ref{lem graphicality propogates} that,
	for any $0 < \rho < r$, $V_{t} - x_0$ is a graph over $\C$ on the region $A_{\rho/2 , \rho} \times [ t_0 - C^2 \rho^2, t_0]$ with $C^{1,\alpha}$-norm bounded by $\epsilon > 0$.
	The statement now follows by taking a union over $\rho \in (0, r)$.
\end{proof}

\section{Pinching Hardt-Simon Minimal Surfaces} \label{sect Pinching Hardt-Simon Minimal Surfs}

Throughout this section, we restrict to the case where $(M_t^n)_{t \in [-T, 0)}$ is a smooth mean curvature flow of properly embedded hypersurfaces in an open subset $U \subset \R^{n+1}$.
$\cl{M} = \bigcup_{t \in [-T, 0)} M_t \times \{ t \} \subset \R^{n+1} \times \R$ denotes its space-time track.
Fix a regular cone $\C^n_0 \subset \R^{n+1}$ and let $\cl{C} = \{ A \cdot \C_0 : A \in O(n+1) \}$ denote all rotations of the cone.
We generally use $\C$ to denote a rotation of the cone $\C_0$, that is $\C \in \cl{C}$.

The goal of this section is to prove Theorem \ref{meta thm 3}, 
which says that if $\cl{M}$ has bounded $H$ and develops a singularity with tangent flow given by an area-minimizing quadratic cone, then there is a type II blow-up limit given by a Hardt-Simon minimal surface (see Subsection \ref{subsect finding Hardt-Simon surfs} for the relevant definitions). 

\subsection{Flows Near a Regular Cone}

We begin with general results for mean curvature flows locally close to a regular cone $\C \in \cl{C}$.
The approach here was inspired by \cite[Section 8]{LSS22}.

\begin{definition}
	We say $\cl{M}$ is \emph{$\epsilon$-close to $\C \in \cl{C}$}  if 
	$\cl{M}$ is a $C^{1, \alpha}$-graph on $\C \cap A_{\epsilon, \epsilon^{-1} } \times [ - \epsilon^{-2}, - \epsilon^2 ]$ with $C^{1, \alpha}$ norm at most $\epsilon$.
	In other words,
	$$M_{t}   \cap A_{\epsilon, \epsilon^{-1} } = G_{\epsilon, \epsilon^{-1} } ( u( \cdot, t) ) \qquad \forall t \in [-\epsilon^{-2} , - \epsilon^2 ]$$
	for some $u : \C \cap A_{\epsilon, \epsilon^{-1} } \times [ - \epsilon^{-2}, - \epsilon^2 ] \to T^\perp \C$ with
	$$\| u \|_{C^{1, \alpha}  (\C \cap A_{\epsilon, \epsilon^{-1} } \times [ - \epsilon^{-2}, - \epsilon^2 ]) } \le \epsilon.$$
	
	We say $\cl{M}$ is \emph{$\epsilon$-close to $\C \in \cl{C}$ at $X = (x, t)$} if $\cl{M}- X$ is $\epsilon$-close to $\C$.
\end{definition}

Throughout the remainder of this section, $\epsilon$ is always assumed to be less than some small constant $\epsilon_0 = \epsilon_0 ( n, \C_0)$ that depends only on the regular cone $\C_0$ and implicitly its dimension $n$.

\begin{definition}
	Suppose $\cl{M}$ is $\epsilon$-close to $\C \in \cl{C}$ at $X$ for some $\epsilon \le \epsilon_0$.
	Define
		$$1 \in [ \lambda_* ( X) , \lambda^*(X) ] \subset [ 0, \infty ]$$
	to be the largest interval such that,
	for all $\lambda \in [ \lambda_*(X), \lambda^*(X) ]$,
	$\cl{D}_{\lambda^{-1}} ( \cl{M} - X )$ is $\epsilon$-close to some $\C' = \C'_\lambda \in \cl{C}$ at $( \mathbf 0, 0)$.
\end{definition}

Occasionally, we may write $\lambda_*(X; \cl{M}, \epsilon)$ or $\lambda^*(X; \cl{M}, \epsilon)$ to emphasize the dependence on $\cl{M}$ and $\epsilon$.
Observe that $\lambda_*(X)$ and $\lambda^*(X)$ are continuous in the basepoint $X$.

We note the following consequence of pseudolocality for mean curvature flow in our setting.

\begin{lem} \label{lem curv bounds}
	Define constants
		$$0 < c_{\C_0} \doteqdot \frac{1}{2} \sup_{x \in L ( \C_0 ) }  | A_{\C_0}|(x)
		< C_{\C_0} \doteqdot 2 \sup_{x \in L ( \C_0 ) }  | A_{\C_0}|(x) < \infty. $$
	There exists $C = C(n, \C_0) \gg 1$ and $\epsilon_0 = \epsilon_0(n, \C_0) \ll 1$ such that  
	if $\cl{M}$ is $\epsilon$-close to $\C \in \cl{C}$ at $X= (x,t)$ for some $\epsilon \le \epsilon_0$,
	then
	\begin{equation} \label{eqn curv bounds}
		\frac{c_{\C_0}}{\rho} \le \sup_{|y-x| = \rho} |A_{M_t}|(y) \le \frac{C_{\C_0}}{\rho}
		\qquad \forall \, (\epsilon + C) \lambda_* (X) \le \rho \le  (\epsilon^{-1} - C )  \lambda^* (X) .
	\end{equation}
\end{lem}
	\begin{proof}
		Let $\lambda \in [ \lambda_* (X) , \lambda^*(X) ]$ and consider $\cl{M}' \doteqdot \cl{D}_{\lambda^{-1}} ( \cl{M} - X )$.
		By definition $\cl{M}'$ can be written as a $C^{1,\alpha}$-graph over some $\C' \in \cl{C}$ on $A_{\epsilon, \epsilon^{-1}} \times [- \epsilon^{-2}, - \epsilon^2 ]$, and the $C^{1, \alpha}$-norm is bounded by $\epsilon \le \epsilon_0$.
		
		Let $\delta = \delta(n, \C_0) \ll 1$ denote some small constant to be determined which depends only on $n$ and $\C_0$.
		If $\epsilon_0 \ll 1$ is sufficiently small (depending also on $\delta$) and $C \gg 1$ is sufficiently large (depending also on $\delta$),
		then pseudolocality for mean curvature flow \cite[Theorem 1.5]{INS19} implies that $\cl{M}'$ can be written as a Lipschitz graph over $\C'$ on $A_{\epsilon + C, \epsilon^{-1} - C} \times [ - \epsilon^{-2}, 0]$.
		Namely, if $M'_t$ denotes the time $t$ time-slice of $\cl{M}'$, then
			$$M'_t \cap A_{\epsilon + C, \epsilon^{-1} - C}
			= G_{\epsilon+C, \epsilon^{-1} - C} (u ( \cdot , t) )
			\qquad \forall t \in [-\epsilon^{-2},0]$$
		for some $u : \C' \cap A_{\epsilon + C, \epsilon^{-1} - C} \times [-\epsilon^{-2}, 0] \to T^\perp \C'$ with
		\begin{multline*}
		 \| u(\cdot, t) \|_{C^{0,1}( \C' \cap A_{\epsilon+C, \epsilon^{-1} - C} ) } \\
		\le C' \left( 
		\sup_{x \in \C' \cap A_{\epsilon + C, \epsilon^{-1} - C} } \frac{|u(x,t)|}{|x|} + \sup_{x \ne y \in \C' \cap A_{\epsilon + C, \epsilon^{-1} - C} } \frac{ | u(x, t) - u(y,t)|}{|x-y|} \right)  \le C' \delta \\ \forall t \in [-\epsilon^{-2}, 0], \text{ (where $C'$ is a universal constant).}
		\end{multline*}
		Interior estimates for mean curvature flow \cite{EckerHuisken91} then imply $u$ satisfies $C^2$-estimates on $\C' \cap A_{\epsilon + C + 1, \epsilon^{-1} - C - 1} \times [ - \epsilon^{-2} + 1, 0]$ of the form
		$$\| u \|_{C^2 ( \C' \cap A_{\epsilon + C + 1, \epsilon^{-1} - C - 1} \times [ - \epsilon^{-2} + 1, 0] )} \le C'' \delta $$
		for some constant $C'' = C'' ( n, \C_0 )$.
	
		If $\delta \ll 1$ is sufficiently small depending on $n$ and $\C_0$, then this $C^2$-closeness of $M_0'$ to the cone $\C'$ implies curvature estimates of the form
			$$\frac{c_\C}{ \rho} \le \sup_{|y| = \rho }| A_{M'_0}|(y) \le \frac{C_\C}{\rho}
			\qquad \forall   \epsilon+ C + 1\le  \rho \le \epsilon^{-1} - C -1.$$
		The estimate \eqref{eqn curv bounds} now follows from undoing the dilation $\cl{D}_{\lambda^{-1}}$ and translation and letting $\lambda$ vary in $[ \lambda_*(X), \lambda^*(X)]$.
		Note that, since $\delta = \delta(n, \C_0)$, the dependence of $\epsilon_0$ and $C$ on $\delta$ can instead be regarded as a dependence on $n$ and $\C_0$.
	\end{proof}

\begin{lem} \label{lem lambda > 0}
	There exists $\epsilon_0 = \epsilon_0(n, \C_0)$ such that the following holds for all $\epsilon \le \epsilon_0$:
	if $(M_t^n)_{t \in [-T, 0)}$ is smooth in $U \subset \R^N$ and its space-time track $\cl{M}$ is $\epsilon$-close to $\C \in \cl{C}$ at $X = (x,t) \in U \times (-\infty, 0)$, then $\lambda_*(X) > 0$.
\end{lem}	
\begin{proof}
	If not, \eqref{eqn curv bounds} implies
		$$|A_{M_t}|(x) = \lim_{\rho \searrow 0} \sup_{| y - x|= \rho} |A_{M_t}|(y) = +\infty,$$
	which contradicts the smoothness of the flow at $(x,t)$.	
\end{proof}

Next, we show that $\lambda_*( x , t)$ satisfies a strict inequality in an annulus.

\begin{lem} \label{lem almost ! minimizer}
	If $\epsilon_0 = \epsilon_0(n, \C_0) \ll 1$ is sufficiently small (depending on $n, \C_0$), then 
	there exists a constant $C = C(n, \C_0)$ such that
	$$\lambda_*( y, t) > \lambda_* ( x, t ) 	\qquad 
	\forall y \in A_{ \lambda_*(x,t) [ \epsilon + C ] , \lambda^*(x, t) [ \epsilon^{-1} - C] }(x).$$
\end{lem}
\begin{proof}
	Fix $C = C(n, \C_0) \ge 1$ as in Lemma \ref{lem curv bounds}. 
	Assume $\epsilon_0 = \epsilon_0 ( n, \C_0) \ll 1$ is small enough so that Lemma \ref{lem curv bounds} holds.
	Let $\epsilon \le \epsilon_0$.
	To simplify notation, denote
	\begin{align*}
		r_0 &\doteqdot ( \epsilon + C) \lambda_*(x, t), 	&
		R_0 &\doteqdot ( \epsilon^{-1} - C) \lambda^*(x, t), 	\\
		r_1 &\doteqdot ( \epsilon + 100 C) \lambda_*(x, t), 	&
		R_1 &\doteqdot ( \epsilon^{-1} - 100C) \lambda^*(x, t). 	
	\end{align*}
	Let $y \in A_{r_1, R_1}(x)$.
	Suppose for the sake of contradiction that $\lambda_*(y, t) \le \lambda_*(x, t)$.
	Then
		$$( \epsilon + C) \lambda_*(y, t) \le r_0 \le ( \epsilon^{-1} - C) \lambda^* (y, t)$$
	if $\epsilon_0(n, \C_0) \ll 1$.
	Lemma \ref{lem curv bounds} therefore gives curvature estimates of the form
	\begin{equation} \label{eqn lower curv bound}
		\frac{c_{\C_0}}{r_0} \le \sup_{|y'-y| = r_0} |A_{M_t}|(y').	
	\end{equation}

	Observe also that the sphere $\partial B( y, r_0)$ is contained in the annulus $A_{r_0, R_0}(x)$ based at $x$.
	Indeed, $|y' -y| = r_0$ implies
	\begin{align*}
		|y' - x|
		&\ge | y - x| - |y' - y|	\\
		&> r_1 - r_0	\\
		&\ge ( 1 + 98 C )\lambda_*(x, t) \\
		&> ( \epsilon + C) \lambda_*(x, t) \\
		&= r_0,
	\end{align*}
	and an analogous argument applies to show $|y' - y|= r_0$ implies $|y' - x| < R_0$.
	Therefore, Lemma \ref{lem curv bounds} based at $x$ also applies and gives
	\begin{align*}
		\sup_{|y' - y| = r_0} |A_{M_t}| (y') 
		&\le \sup_{|y' - y| = r_0} \frac{ C_{\C_0}}{ | y' - x|} \\
		&\le \sup_{|y' - y| = r_0} \frac{ C_{\C_0}}{  | y - x| - | y' - y| }	\\
		&< \frac{ C_{\C_0}}{ r_1 - r_0} \\
		&< \frac{ C_{\C_0}}{4 r_0}	\\
		&= \frac{ c_{\C_0}}{r_0} ,
	\end{align*}
	which contradicts estimate \eqref{eqn lower curv bound}.
	This completes the proof after relabelling $100 C$ to $C$.	
\end{proof}

\subsection{Finding Hardt-Simon Minimal Surfaces} \label{subsect finding Hardt-Simon surfs}

Assume additionally throughout this subsection that $\cl{M}$ has $H \in L^\infty L^\infty_{loc}(U \times [-T, 0) ) $ and that $\mathbf 0 \in U$.
Take a sequence $\lambda_i \searrow 0$ and suppose
	$$\cl{M}_i \doteqdot \cl{D}_{\lambda_i^{-1} } \cl{M} \rightharpoonup \cl{M}_{\C_0}$$
where $\cl{M}_{\C_0} = \C_0 \times (-\infty, 0)$ is the flow of the stationary cone  $\C_0$.

Since $\cl{M}_i \rightharpoonup \cl{M}_{\C_0}$, $\cl{M}_i$ is $\epsilon$-close to $\C_0$ at $(\mathbf 0, -1)$ for $i \gg 1$ and
$\lim_{i \to \infty} \lambda_*( \mathbf 0, -1; \cl{M}_i )= 0$.
Let $R_0 = C(n, \C_0)$ be the constant from Lemma \ref{lem almost ! minimizer}
and obtain $x_i \in \ol{B}_{2 R_0} ( \mathbf 0 )$ such that 
	$$\lambda_* (x_i, -1 ; \cl{M}_i ) = \min_{x \in \ol{B}_{2 R_0} ( \mathbf 0 )} \lambda_*(x, -1; \cl{M}_i ).$$
Denote $X_i = (x_i, -1)$ and observe $\lim_{i \to \infty} \lambda_*(X_i; \cl{M}_i)= \lim_{i \to \infty} \lambda_*(\mathbf 0 , -1; \cl{M}_i) = 0$.
Moreover, Lemma \ref{lem almost ! minimizer} implies that the minimizer $x_i$ must lie in $\ol{B}_{\lambda_*(\mathbf 0, -1; \cl{M}_i) [ \epsilon + C ] }( \mathbf 0)$, which implies $\lim_{i \to \infty} x_i = \mathbf 0$.

Define
	$$\cl{M}'_{i} \doteqdot \cl{D}_{\frac{1}{ \lambda_* ( X_i; \cl{M}_i ) } } ( \cl{M}_i - X_i ).$$
By construction, for all $i$, $\cl{M}_{i}'$ is $\epsilon$-close to some $\C'_{i} \in \cl{C}$ at $(\mathbf 0 , 0)$.

\begin{remark}
	In what follows, we deviate from the notation of subsection \ref{subsect Huisken's mono formula} and write Huisken's monotonic quantity as
	\begin{align*}
		\Theta( (x,t), \cl{M}, r ) &= \int \Phi_{x,t}( \cdot, t-r^2) d \mu_{t - r^2},  \\
		\Theta( (x,t), \cl{M} ) &= \lim_{r \searrow 0}  \int \Phi_{x,t}( \cdot, t-r^2) d \mu_{t - r^2}
	\end{align*}
	where $\cl{M} = ( \mu_t)_{t \in (a,b)}$ is a Brakke flow in $\R^{n+1}$ and $r>0$.
\end{remark}

\begin{lem} \label{lem ancient limit}
	Along some subsequence $i \to \infty$,
	$$\cl{M}_{i}' \rightharpoonup \hat{\cl{M}} \doteqdot ( \hat \mu_t )_{t \in \R} $$
	where $\hat{\cl{M}}$ is an eternal integral $n$-Brakke flow in $\R^{n+1}$ 
	such that 
	\begin{enumerate}
		\item \label{item ancient entropy bound} (entropy bound) for all $r > 0$ and $(x,t) \in \R^{n+1} \times \R$
			$$\Theta((x,t), \hat{\cl{M}}, r ) \le \Theta ( (\mathbf 0, 0), \cl{M}_{\C_0} ),$$
		\item \label{item ancient timeslices} for all $t \in \R$, $\hat \mu_t = \mu_{\hat V_t}$ for some stationary integral $n$-varifold $\hat V_t$,
		\item \label{item ancient H bound} $\| H \|_{L^\infty L^\infty ( \R^{n+1} \times \R ) } = 0$, and
		\item \label{item ancient close to cone} for all $\lambda \in [1, \infty)$, $\cl{D}_{\lambda^{-1}} \hat{\cl{M}}$ is $\epsilon$-close to some $\hat{\mathbf C}_{\lambda} \in \cl{C}$.
	\end{enumerate}
\end{lem}
\begin{proof}
	For any $r > 0$ and $(x,t) \in \R^{n+1} \times \R$,
	\begin{equation} \label{eqn entropy bound}
	\begin{aligned}
		\limsup_{i \to \infty} \Theta( (x, t), \cl{M}'_{i}, r )
		&=\limsup_{i \to \infty} \Theta( (x,t), \cl{D}_{\lambda_*(X_i)^{-1}} ( \cl{M}_i - X_i ), r ) \\
		&= \limsup_{i \to \infty} \Theta( X_i + (x \lambda_*(X_i), t \lambda_*(X_i)^2), \cl{M}_i , r \lambda_*(X_i)  )  \\
		&\le \Theta ( (\mathbf 0, -1), \cl{M}_{\C_0} , 0)
	\end{aligned}
	\end{equation}
	by the limiting behavior of $X_i, \cl{M}_i, \lambda_*(X_i)$ and the upper semi-continuity of $\Theta$.
	Huisken's monotonicity formula therefore implies that the $\cl{M}'_{i}$ satisfy local uniform area bounds. 
	Compactness of Brakke flows with local uniform area bounds then allows us to extract a subsequential limit $\hat{\cl{M}}$ of the $\cl{M}'_{i}$.
	Observe that, since $\lim_{i \to \infty} \lambda_*(X_i) = 0$, the limiting integral $n$-Brakke flow $\hat{\cl{M}}$ is defined for all $t \in \R$.
	Additionally, \eqref{eqn entropy bound} shows that for any $r > 0$ and $(x,t) \in \R^{n+1} \times \R$
		$$\Theta( (x,t), \hat{\cl{M}}, r) = \lim_{i \to \infty} \Theta( (x,t), \cl{M}_{i}', r) \le \Theta( (\mathbf 0, -1) , \cl{M}_{\C_0}, 0) = \Theta ( (\mathbf 0, 0) , \cl{M}_{\C_0}, 0 ).$$

	The fact that $\cl{M}$ has mean curvature $H \in L^\infty L^\infty_{loc}( U \times [-T, 0) )$ implies that the mean curvature $H_i$ of the rescalings $\cl{M}'_{i}$ has $\lim_{i \to \infty} \| H_i \|_{L^\infty L^\infty ( K \times [a,b] )} = 0$ for any $K \times [a,b] \Subset \R^{n+1} \times \R$.
	Thus, the limiting Brakke flow $\hat{\cl{M}}$ has $\| H \|_{L^\infty L^\infty( \R^{n+1} \times \R )} =0$.
	
	Let $t_0 \in \R$ and consider the $t_0$-timeslice $\cl{M}_i'(t_0)$ of $\cl{M}_i'$.
	Lemma \ref{Lem Compactness} applies to give that some subsequence $\cl{M}_i'(t_0)$ converges in the weak varifold sense to a stationary integral $n$-varifold $\hat V_{t_0}$.
	In particular, the underlying measures converge $\mu_{\cl{M}_i'(t_0)} \rightharpoonup \mu_{\hat V_{t_0}}$.
	On the other hand, the Brakke flow convergence $\cl{M}_i' \rightharpoonup \hat{\cl{M}}$ implies $\mu_{ \cl{M}'_i(t_0)} \rightharpoonup \hat \mu_{t_0}$ and so $\hat \mu_{t_0} = \mu_{\hat V_{t_0}}$.
	
	Let $\lambda \in [1, \infty)$.
	Observe that $ \lim_{i\to \infty} \lambda_*( X_i; \cl{M}_i) = 0$ and $\lambda^* ( X_*; \cl{M}_i ) \ge 1$ for all $i \gg 1$.
	Thus, $\lambda \lambda_* ( X_i; \cl{M}_i ) \in \left[ \lambda_*( X_i ; \cl{M}_i) , \ \lambda^* ( X_i ; \cl{M}_i ) \right]$ for all $i \gg 1$.
	It follows that $\cl{D}_{\lambda^{-1} } \cl{M}_{i}' = \cl{D}_{\frac{1}{ \lambda \lambda_* (X_i ; \cl{M}_i) } } ( \cl{M}_i - X_i )$ is $\epsilon$-close to some $\C_{\lambda, i}' \in \cl{C}$ for all $i \gg 1$.
	After passing to a subsequence and using that $\cl{D}_{\lambda^{-1}} \cl{M}_{i}' \rightharpoonup \cl{D}_{\lambda^{-1}} \hat{\cl{M}}$, it follows that there exists a limiting cone $\hat \C_\lambda \in \cl{C}$ such that $\cl{D}_{\lambda^{-1}} \hat{\cl{M}}$ is $\epsilon$-close to $\hat \C_\lambda$.
	Since $\lambda \in [1, \infty)$ was arbitrary, this completes the proof.
\end{proof}

\begin{definition}
	For $p, q \in \N$ with $p+q = n-1$, define the $n$-dimensional \emph{quadratic minimal cone} or \emph{generalized Simons cone} $\C^{p,q}$ to be the hypersurface
		$$\C^{p,q} \doteqdot \{ (x,y) \in \R^{p+1} \times \R^{q+1} : q|x|^2 = p |y|^2 \} \subset \R^{n+1}.$$
\end{definition}

\begin{remark}
	$\C^{p,q}$ is minimal for any $p,q$.
	Moreover, $\C^{3,3}, \C^{2,4} \text{ (equivalently } \C^{4,2}),$ and $\C^{p,q}$ for any $p+q >6$ are all area minimizing.
	In fact, these are the only quadratic minimal cones which are area minimizing (see e.g. \cite{Simoes74} and references therein).
\end{remark}

\begin{definition}[Hardt-Simon foliation] \label{defn Hardt-Simon foliation}
	Let $\C^{p,q}$ be a quadratic minimal cone which is area minimizing.
	Denote the two connected components of $\R^{n+1} \setminus \C^{p,q}$ by $E_{\pm}$. 
	\cite{HS85} showed that there exist smooth, minimal surfaces $S_{\pm} \subset E_{\pm}$ which are both asymptotic to $\C^{p,q}$ at infinity and whose dilations $\lambda S_{\pm}$ $(\lambda > 0)$ foliate $E_{\pm}$, respectively.
	By dilating, we can assume without loss of generality that both $S_\pm$ are normalized to have $\dist (S_{\pm}, \mathbf 0 ) = 1$.
	
	For any $\lambda \in \R$, define
	\begin{equation} \label{eqn Hardt-Simon foliation}
		S_\lambda \doteqdot  
		\left \{ \begin{aligned}
 			&\lambda S_+, \quad & \text{if } \lambda > 0, \\
 			& \C^{p,q}, & \text{if } \lambda = 0, \\
 			&-\lambda S_-, \quad & \text{if } \lambda < 0. \\
		 \end{aligned} \right.
	\end{equation}
	We refer to this family of minimal hypersurfaces as the \emph{Hardt-Simon foliation}.
\end{definition}

\begin{thm} \label{thm Hardt-Simon limit}
	Let $\C_0 = \C^{p,q} \subset \R^{n+1}$ ($p+q = n-1$) be a generalized Simons cone, and let $\cl{C} = \{ A \cdot \C_0 : A \in O(n+1) \}$ denote all rotations of the cone.
	Let $0 < \epsilon_0 = \epsilon_0(n, \C_0) \ll 1$ be sufficiently small so that Lemmas \ref{lem curv bounds}--\ref{lem almost ! minimizer} hold, and let $0 < \epsilon \le \epsilon_0$.
	
	If $\C_0 = \C^{p,q}$ is area-minimizing,
	then there exist $\lambda_0 \ne 0, A_0 \in O(n+1),$ and $a_0 \in \R^{n+1}$ such that 
	the eternal Brakke flow $\hat{\cl{M}}$ obtained in Lemma \ref{lem ancient limit} is the static flow of the smooth Hardt-Simon minimal surface $M =  A_0 \cdot S_{\lambda_0} + a_0$ for all $t \in \R$.
	
	Additionally, the scale $\lambda_0 \ne 0$ and center $a_0 \in \R^{n+1}$ are such that, for all $0 < \epsilon' < \epsilon$ and all $\C \in \cl{C}$, $\hat{\cl{M}}$ is not $\epsilon'$-close to $\C$.
\end{thm}
\begin{proof}	
	It follows from the entropy bound Lemma \ref{lem ancient limit} \eqref{item ancient entropy bound}, compactness of Brakke flows with uniform local area bounds, and Huisken's monotonicity formula that there exists a limiting shrinker $\cl{D}_{\lambda_i^{-1}} \hat{\cl{M}} \rightharpoonup \hat{\cl{M}}_{-\infty}$ along some sequence $\lambda_i \to \infty$.
	Since $\hat{\cl{M}}$ has $H \equiv 0$ (Lemma \ref{lem ancient limit} \eqref{item ancient H bound}),
	the same argument as in the proof of Lemma \ref{lem tangent flow} shows that, for $t < 0$, $\hat{\cl{M}}_{-\infty}$ must be the flow of a stationary cone $\hat{\C}$ (which is dilation invariant with respect to $\mathbf 0 \in \R^{n+1}$).
	Because $\cl{D}_{\lambda_i^{-1}} \hat{\cl{M}}$ is $\epsilon$-close to some $\hat{\C}_{\lambda_i} \in \cl{C}$ for all $i$ (Lemma \ref{lem ancient limit} \eqref{item ancient close to cone}), we can pass to a further subsequence (still denoted by $i$) so that the $\hat{\C}_{\lambda_i}$ converge to $\hat{\C}_{-\infty} \in \cl{C}$ and $\hat{\cl{M}}_{-\infty} = \cl{M}_{\hat \C}$ is $\epsilon$-close to $\hat{\C}_{-\infty} \in \cl{C}$.
	By the dilation-invariance of the cone $\hat{\C}$, it follows that $\hat \C$ can be written \emph{globally} as a $C^{1, \alpha}$-graph over $\hat{\C}_{-\infty}$ with $C^{1,\alpha}$-norm less than or equal to $\epsilon$.

	Consider the Hardt-Simon foliation \eqref{eqn Hardt-Simon foliation} rotated so that $S_0 = \hat{\C}_{-\infty}$.
	If $\epsilon \ll 1$ is sufficiently small depending on $n, \C_0$, then \cite[Theorem 3.1]{ES19} implies that there exists $a \in \R^{n+1}, q \in SO(n+1), \lambda' \in \R$ such that $\hat{\C}$ is a $C^{1, \beta}$-graph over $a + q ( S_{\lambda'} )$ and the graphing function $u$ satisfies an improved decay estimate
		$$\sup_{x \in B_r(a) \cap (a + q( S_{\lambda'} )) } |u(x)| \lesssim r^{1 + \beta} \qquad \forall r \le 1/2.$$
	Because $\hat{\C}$ is dilation-invariant (with respect to $\mathbf 0$), this estimate is only possible if $a = \mathbf 0, \lambda' = 0$, and $u \equiv 0$.
	In other words, $\hat{\C} = q ( S_0) = q ( \hat{\C}_{-\infty} ) \in \cl{C}$.
	Thus, for $t < 0$, $\hat{\cl{M}}_{-\infty} = \cl{M}_{\hat{\C}}$ and $\hat{\C} \in \cl{C}$.

	Because $\cl{D}_{\lambda_i^{-1}} \hat{\cl{M}} \rightharpoonup \hat{\cl{M}}_{-\infty}$ converges as Brakke flows, we can find $t_0 < 0$ and pass to a further subsequence so that the time slices $\hat{V}_i \doteqdot \lambda_i^{-1} \hat{V}_{t_0 \lambda_i^2} \rightharpoonup \hat{\C}$ converge as integral $n$-varifolds.
	For $i \gg 1$, \cite[Theorem 3.1]{ES19} applies and gives that, for all $i \gg 1$, there exist $a_i\in \R^{n+1}$, $q_i \in SO(n+1)$, $\lambda'_i \in \R$ such that $\spt \hat V_i \cap B_{1/2}$ is a $C^{1,\beta}$-graph over $a_i +q_i(S_{\lambda'_i})$.
	Moreover, after renormalizing the Hardt-Simon foliation \eqref{eqn Hardt-Simon foliation} so that $S_0 = \hat{\C}$, 
		$$\lim_{i \to \infty} \left( |a_i| + |q_i - Id | + |\lambda_i'| \right) = 0.$$
	
	\begin{claim} \label{claim lambda_i' not 0}
		$\lambda_i' \ne 0$ for all $i \gg 1$.	
	\end{claim}
	\begin{proof}[Proof of Claim.]
	Consider some index $i$ with $\lambda_i' = 0$.
	Then the monotonicity formula \eqref{simple mono forla eqn} for stationary varifolds implies that
		$$\theta_{\hat V_i - a_i} ( \mathbf 0) = \lim_{r \searrow 0} \frac{ \mu_{\hat V_i - a_i} ( B_r) }{\omega_n r^n}  \ge \theta_{q_i ( S_0) } ( \mathbf 0) = \theta_{\C_0} ( \mathbf 0 ) = \Theta( \cl{M}_{\C_0}, ( \mathbf 0, 0) ).$$
	Theorem \ref{thm blow-ups} then gives that
	\begin{align*}	
		\theta_{\C_0} ( \mathbf 0 )
		&\le \theta_{\hat V_i - a_i} ( \mathbf 0) \\
		&= \theta_{\hat V_i} ( a_i ) \\
		&\le \Theta_{\cl{D}_{\lambda_i^{-1}} \hat{\cl{M}}} ( a_i , t_0) 
		&& ( \text{Theorem \ref{thm blow-ups}} ) \\ 
		&\le \Theta_{\cl{D}_{\lambda_i^{-1}} \hat{\cl{M}}} (( a_i , t_0), r)
			&& ( \text{Huisken's monotonicity formula \eqref{Huisken's mono forla}, } r >0 ) \\
		&= \Theta_{\hat{\cl{M}}} ( ( \lambda_i a_i, \lambda_i^2 t_0) , \lambda_i r ) \\
		&\le \theta_{\C_0} ( \mathbf 0) 		
			&& ( \text{Lemma \ref{lem ancient limit} \eqref{item ancient entropy bound}}).
	\end{align*}
	Thus, we have equality throughout and 
		$$\Theta_{\hat{\cl{M}}} ( ( \lambda_i a_i, \lambda_i^2 t_0) , r) = \theta_{\C_0} ( \mathbf 0) \qquad \forall r > 0.$$
	It follows from Huisken's monotonicity formula \eqref{Huisken's mono forla} that $\hat{\cl{M}} - ( \lambda_i a_i, \lambda_i^2 t_0 )$ is a shrinker (for $t<0$) and must therefore be equal to the limiting shrinker $\hat{\cl{M}}_{-\infty}$, that is
	\begin{equation} \label{eqn ancient = shrinker}
		\hat{\cl{M}}-( \lambda_i a_i, \lambda_i^2 t_0 ) = \hat{\cl{M}}_{-\infty} = \cl{M}_{\hat{\C}}
		\qquad \text{for } t < 0.
	\end{equation}
	
	Write $\hat{\cl{M}} = ( \hat \mu_t )_{t \in \R}$, and note $\hat{\mu}_t = \mu_{\hat{\C}+\lambda_i a_i}$ for all $t < \lambda_i^2 t_0$ by \eqref{eqn ancient = shrinker}.
	Moreover, Brakke's inequality and the fact that $\hat{\cl{M}}$ has $H \equiv 0$ (Lemma \ref{lem ancient limit} \eqref{item ancient H bound}) implies $\hat{\mu}_{t_1} \ge \hat{\mu}_{t_2}$ for all $t_1 \le t_2$.
	In particular, $\spt \hat{\mu}_{t_2} \subset \spt \hat{\mu}_{t_1}$ for $t_1 \le t_2$.
	
	Let $t \in [\lambda_i^2 t_0 , -\epsilon^2]$ and recall $\hat{\mu}_t$ is represented by a stationary integral $n$-varifold $\hat V_t$ with $H = H_{\hat V_t} = 0$ (Lemma \ref{lem ancient limit} \eqref{item ancient timeslices}).
	Then
		$$\spt \hat V_t = \spt \hat{\mu}_t \subset \spt \hat \mu_{\lambda_i^2 t_0} = \hat{\C} + \lambda_i a_i.$$
	On the other hand, Solomon-White's strong maximum principle \cite{SW89} 
	applied to the smooth manifold $(\hat{\C} + \lambda_i a_i ) \setminus \{ \lambda_i a_i \} \subset \R^{n+1} \setminus \{ \lambda_i a_i \}$
	implies that either
		$$(\hat{\C} + \lambda_i a_i) \setminus \{ \lambda_i a_i \} \subset \spt \hat V_t
		\qquad \text{or} \qquad 
		(\hat{\C} + \lambda_i a_i ) \cap  \spt \hat V_t \setminus \{ \lambda_i a_i \} = \emptyset.$$
	In particular, either
		$$\hat \C + \lambda_i a_i = \spt \hat V_t
		\qquad \text{or} \qquad 
		\spt \hat V_t \subset \{ \lambda_i a_i \}$$
	since $\spt \hat V_t \subset \hat \C + \lambda_i a_i$ and $\spt \hat V_t$ is a closed set.
	However, the second case is impossible since $\hat{\cl{M}}$ is $\epsilon$-close to some $\hat \C_{\lambda=1}  \in \cl{C}$ by Lemma \ref{lem ancient limit} \eqref{item ancient close to cone}.
	Thus, $\spt \hat V_t = \hat \C + \lambda_i a_i$ for all $t \in [\lambda_i^2 t_0, -\epsilon^2 ]$.
	By the entropy bounds Lemma \ref{lem ancient limit} \eqref{item ancient entropy bound} and the constancy theorem \cite[Ch. 8, \S 4]{SimonGMT}, it follows that in fact $\hat V_t = \hat{\C} + \lambda_i a_i$ for all $t \in [\lambda_i^2 t_0, -\epsilon^2 ]$.
	Thus, in combination with \eqref{eqn ancient = shrinker}, we have shown
	\begin{equation} \label{eqn ancient = shrinker +}
		\hat{\cl{M}} - ( \lambda_i a_i , 0) = \cl{M}_{\hat \C} \qquad \text{for } t \in [\lambda_i^2 t_0, -\epsilon^2]
	\end{equation}
	for any index $i$ such that $\lambda_i' = 0$.
		
	Suppose for the sake of contradiction that there exist arbitrarily large indices $i$ with $\lambda_i' = 0.$
	Then there exists $i_0$ such that $\lambda_{i_0}' = 0$ and $\lambda_{i_0}^2 t_0 \le - \epsilon^{-2}-1$.	
	Hence, equality \eqref{eqn ancient = shrinker +} holds for $t \in [-\epsilon^{-2}-1, - \epsilon^2]$ and 
	\begin{align*}
		&\cl{D}_{\lambda_*( X_j; \cl{M}_j)^{-1}} ( \cl{M}_j - X_j - ( \lambda_*( X_j) \lambda_{i_0} a_{i_0} , 0 ) ) \\
		={}&  \cl{M}_{j}' - ( \lambda_{i_0} a_{i_0} , 0) \\
		\rightharpoonup{}& \hat{\cl{M}} - ( \lambda_{i_0} a_{i_0} , 0)  && (\text{as } j \to \infty) \\
		={}& \cl{M}_{\hat \C} && ( \text{for } t \in [-\epsilon^{-2}-1, -\epsilon^2] ).
	\end{align*}
	In particular, White's local regularity theorem \cite{White05}
	implies $\cl{M}_j' - ( \lambda_{i_0} a_{i_0} , 0)$ converges to $\cl{M}_{\hat \C}$ in $C^\infty_{loc}( \R^{n+1} \setminus \{ \mathbf 0 \} \times [- \epsilon^{-2}-1, -\epsilon^2])$ as $j \to \infty$.
	Using smoothness of the flows $\cl{M}_j'$, it then follows that, for some large enough $j$, $\cl{D}_{\lambda^{-1}} ( \cl{M}_j' - ( \lambda_{i_0} a_{i_0} , 0) )$ is $\epsilon$-close to $\hat \C$ at $(\mathbf 0, 0)$ for all $\lambda$ in a neighborhood of $1$.
	Equivalently, $\cl{D}_{\lambda^{-1}} ( \cl{M}_j - X_j - ( \lambda_*(X_j) \lambda_{i_0} a_{i_0}, 0))$ is $\epsilon$-close to $\hat \C \in \cl{C}$ at $(\mathbf 0, 0)$ for all $\lambda$ in a neighborhood of $\lambda_*(X_j)$.
	This however contradicts the definitions of $X_j$ and $\lambda_*$, and this contradiction completes the proof of the claim. 
	\end{proof}
	
	With Claim \ref{claim lambda_i' not 0} in hand, $\hat V_i \cap B_{1/2}$ is a $C^{1,\beta}$-graph over a \emph{smooth} minimal surface $a_i + q_i (S_{\lambda'_i})$ ($\lambda_i' \ne 0$) for all $i \gg 1$.
	Interior regularity for stationary varifolds then implies that $\hat V_i \cap B_{1/4}$ is smooth for all $i \gg 1$, and thus so is $\spt \hat \mu_{t_0 \lambda_i^2} \cap B_{\lambda_i /4}$.
	
	Let $t \le - \epsilon^2$ and $R > \epsilon^2$.
	Let $i \gg 1$ be sufficiently large such that $\spt \hat \mu_{t_0 \lambda_i^2} \cap B_{\lambda_i/4}$ is smooth, $t_0 \lambda_i^2 < t$, and $\lambda_i/4 > R$.
	Since $\hat \mu_t \le \hat \mu_{t_0 \lambda_i^2}$, it follows that 
		$$\spt \hat V_t \cap B_R = \spt \hat \mu_t \cap B_R \subset \spt \hat \mu_{t_0 \lambda_i^2} \cap B_R.$$
	Because $\spt \hat \mu_{t_0 \lambda_i^2} \cap B_R$ is smooth, the strong maximum principle \cite{SW89} applies and implies that either
		$$\spt \hat V_t \cap B_R = \spt \hat \mu_{t_0 \lambda_i^2} \cap B_R \qquad \text{or} \qquad \spt \hat V_t \cap B_R = \emptyset.$$
	However, the second case is impossible, since it would imply (together with $\hat \mu_{-\epsilon^2} \le \hat \mu_t$) that $\spt \hat \mu_{-\epsilon^2}\cap B_R = \emptyset$, which contradicts that $\hat{\cl{M}}$ is $\epsilon$-close to some $\hat{\C}_{\lambda = 1} \in \cl{C}$ (Lemma \ref{lem ancient limit} \eqref{item ancient close to cone}).
	Hence, $\spt \hat V_t \cap B_R = \spt \hat \mu_{t_0 \lambda_i^2} \cap B_R$.
	By letting $i,R$, and $t$ vary, it follows that there exists a smooth manifold $M$ such that 
		$$M = \spt \hat \mu_t = \spt \hat V_t \qquad \text{for all } t \le -\epsilon^2.$$
	
	Additionally, $M$ is minimal ($H_M \equiv 0$) and 
		$$\lambda_i^{-1} M \rightharpoonup \hat \C \in \cl{C} \qquad ( \text{as varifolds})$$
	since $\cl{D}_{\lambda_i^{-1}} \hat{\cl{M}} \rightharpoonup \cl{M}_{\hat \C}$ as Brakke flows.
	The proof of \cite[Corollary 3.7]{ES19} holds in this setting and implies that $M$ is necessarily a smooth Hardt-Simon minimal surface, that is $M = q(S_\lambda) + a$ for some $q \in SO(n+1)$, $\lambda \ne 0$ and $a \in \R^{n+1}$.
	It then follows from the constancy theorem \cite[Ch. 8, \S 4]{SimonGMT}
	and the entropy bounds for $\hat{\cl{M}}$ (Lemma \ref{lem ancient limit} \eqref{item ancient entropy bound}) that 
	$\hat V_t = M = q( S_\lambda) + a$ for all $t \le - \epsilon^2$.
	In summary, $\hat{\cl{M}}$ is the stationary flow of a smooth Hardt-Simon minimal surface $M = q( S_\lambda) + a$ for all $t \le -\epsilon^2$.
		
	Using the strong maximum principle \cite{SW89} and the fact that $\hat \mu_{t_1} \ge \hat \mu_{t_2}$ for $t_1 \le t_2$, it can be shown that there exists $T \in [-\epsilon^2, +\infty]$ such that the flow $\hat{\cl{M}} = ( \mu_{\hat V_t} )_{t \in \R}$ has
		$$\hat V_t = \left\{ \begin{aligned}
 			&M && \text{for } t < T, \\
 			&\emptyset && \text{for } t > T.
		 \end{aligned}\right.$$
	Since $\hat{\cl{M}}$ is a limit of smooth flows $\cl{M}_i'$, $\hat{\cl{M}}$ is unit-regular \cite[Theorem 4.2]{SW20} and therefore $T= +\infty$.
	Thus, $\hat{\cl{M}} = ( \mu_{M} )_{t \in \R}$ is the stationary flow of the smooth Hardt-Simon minimal surface $M = q(S_\lambda) + a$.
	In particular, White's local regularity theorem \cite{White05} implies $\cl{M}'_{i}$ converges to $\hat{\cl{M}}$ in $C^\infty_{loc}( \R^{n+1} \times \R )$ as $i \to \infty$.
	
	Finally, suppose for the sake of contradiction that there exists $0 < \epsilon' < \epsilon$ and $\C \in \cl{C}$ such that $\hat{\cl{M}}$ is $\epsilon'$-close to $\C$.
	Since $\cl{M}_i'$ converges to $\hat{\cl{M}}$ in $C^\infty_{loc}( \R^{n+1} \times \R)$, 
	one can take $\epsilon'' \in ( \epsilon', \epsilon)$ and deduce that $\cl{M}_i'$ is $\epsilon''$-close to $\C$ for $i \gg 1$ sufficiently large.
	Since $\cl{M}_i'$ is smooth and $\epsilon'' < \epsilon$, $\cl{D}_{\lambda^{-1}} \cl{M}_i' = \cl{D}_{\lambda^{-1} \lambda_*(X_i; \cl{M}_i)^{-1}} ( \cl{M}_i - X_i) $ must then be $\epsilon$-close to $\C$ for all $\lambda$ in a neighborhood of $1$.
	This, however, contradicts the definition of $\lambda_*(X_i; \cl{M}_i)$.
\end{proof}

Theorem \ref{thm Hardt-Simon limit} and the results of this section complete the proof of Theorem \ref{meta thm 3}.

\appendix

\section{Varifolds with Generalized Mean Curvature $H \in L^p_{loc}$} \label{appendix Varifolds with H in L^p}

In this appendix, we collect some standard results for varifolds with generalized mean curvature $H \in L^p_{loc}$ that are cited throughout the article.
For example, the compactness statement Lemma \ref{Lem Compactness} and monotonicity formula Proposition \ref{prop simple mono forla} are given here.

\begin{lem}[Compactness] \label{Lem Compactness}
	Let $2 \le n < N$, $U \subset \R^N$ be open, and $p \in (1, \infty]$.
	The collection of integer rectifiable $n$-varifolds in $U$ with locally uniformly bounded area and locally uniformly bounded generalized mean curvature $H \in L^p_{loc}(U)$ is weakly compact.
	
	In other words, if $V_i$ is a sequence of integral $n$-varifolds in $U$ such that
	\begin{enumerate}
		\item the $V_i$ have locally uniformly bounded area, i.e. 
			$$\sup_i \mu_{V_i} ( K ) < \infty \qquad (\forall K \Subset U),$$
			and
		\item the $V_i$ have generalized mean curvature $H_i \in L^p_{loc}(U)$
		with uniform $L^p_{loc}(U)$ bounds, i.e.
			$$\sup_i \|  H_i \|_{L^p(K, d \mu_{V_i} ) } < \infty \qquad (\forall K \Subset U),$$
	\end{enumerate}
	then there exists a subsequence $V_{i_j}$ and an integral $n$-varifold $V_\infty$ in $U$ such that 
	$V_{i_j} \rightharpoonup V_\infty$ weakly as varifolds.
	
	Moreover, $V_\infty$ has 
		$$\mu_{V_\infty} (K ) \le \limsup_i \mu_{V_i} (K)
		\qquad (\forall K \Subset U )$$
	and generalized mean curvature $H_\infty \in L^p_{loc}(U)$
	with
		$$\|  H_\infty \|_{L^p ( K , d \mu_{V_\infty} ) } 
		\le \limsup_i \|  H_i \|_{L^p ( K , d \mu_{V_i} ) } \qquad (\forall K \Subset U ).$$
\end{lem}

\begin{proof}
	Allard's compactness theorem \cite{Allard72} gives compactness under the weaker assumption where (2) is replaced by the bound
		$$ \sup_i | \delta V_i (X) | \le C_K \| X \|_{C^0(K)} 
		\qquad \forall X \in C^1_c ( K , \R^N ), \, K \Subset U. $$
	Therefore, there exists a subsequential limit $V_{i_j} \rightharpoonup V_\infty$ 
	such that $V_\infty$ is an integral $n$-varifold 
	with locally finite area and  the weaker property that 
		$$| \delta V_\infty ( X) | \le C_K \| X \|_{C^0} 
		\qquad \forall X \in C^1_c ( K  , \R^N ), \, K \Subset U.$$
	
	However, convergence as varifolds $V_{i_j} \rightharpoonup V_\infty$ implies  
	that for any $K \subset U$ compact and $X \in C^1_c(K, \R^N)$
	\begin{gather} \label{limiting H bound eqn}
	\begin{aligned}
		&| \delta V_\infty ( X) | \\
		={}& \lim_{j \to \infty} 	| \delta V_{i_j} ( X) |		\\
		={}& \lim_{j \to \infty} \left| \int H_{i_j} \cdot X d \mu_{V_{i_j} } \right|	\\
		\le{}& \left \{ 
		\begin{aligned} 
			&\left(  \limsup_{i \to \infty} \| H_i \|_{L^p_{loc} ( K, d \mu_{V_i} ) } \right)  \left( \int |X|^{ \frac{p}{p-1} } d \mu_{V_\infty} \right)^{\frac{p-1}{p}} 
			&& \text{if } p \in (1, \infty),\\
			& \left(  \limsup_{i \to \infty} \| H_i \|_{L^\infty_{loc} ( K, d \mu_{V_i} ) } \right)  \int |X| d \mu_{V_\infty} 
			&& \text{if } p = \infty.
		\end{aligned} \right.
	\end{aligned}
	\end{gather}
	The Radon-Nikodym theorem then implies that $\delta V_\infty = - H_\infty \in L^p_{loc} ( U, d\mu_{V_\infty} )$.
	\eqref{limiting H bound eqn} then also implies
		$$\|  H_\infty \|_{L^p(K, d \mu_{V_\infty} )} 
		\le \limsup_i \|  H_i \|_{L^p ( K , d \mu_{V_i} ) }.$$
\end{proof}

\begin{lem} \label{Lem Checking Convergence}
	Let $2 \le n < N$, $U \subset \R^N$ be open, and $p \in (1, \infty]$.
	Let $(V_i)_{i \in \mathbb{N} \cup \{ \infty \} }$ be a collection of integer rectifiable $n$-varifolds with generalized mean curvature $H_i \in L^p_{loc}(U)$.
	Assume
		$$\sup_{i} \mu_{V_i} (K) + \sup_i \| H_i \|_{L^p(K, d\mu_{V_i} ) } < \infty \qquad (\forall K \Subset U).$$
	
	If 
	\begin{equation} \label{super weak convergence}
		\int f d \mu_{V_i} \xrightarrow{i \to \infty} \int f d \mu_{V_\infty} 
		\qquad \forall f \in C^\infty_c ( \R^N) \text{ with }  f \ge 0,
	\end{equation}
	then $V_i \rightharpoonup V_\infty$ as varifolds as $i \to \infty$.
\end{lem}
\begin{proof}
	We first show $d \mu_{V_i} \rightharpoonup d\mu_{V_\infty}$.
	
	Let $f \in C^0_c ( U )$.
	There exists $\delta > 0$ and a compact set $K \subset U$ such that 
		$$\supp f \subset B_\delta (\supp f) \subset K \subset U$$
	where $B_\delta ( \supp f)$ denotes the radius $\delta$ neighborhood of $\supp f$.
	Denote
		$$ C_K \doteqdot \sup_i \mu_{V_i} ( K ) 	< \infty.$$

	Let $\epsilon > 0$.
	Split $f = f_+ - f_-$ into positive and negative parts.
	Convolving $f_+, f_-$ with a suitable mollifiers, we can find $\tilde f_+, \tilde f_- \in C^\infty_c ( U )$ such that
		$$\tilde f_{\pm} \ge 0, \qquad
		\supp \tilde f_{\pm} \subset K, \qquad \text{ and } \qquad
		\| \tilde f_{\pm}  - f_{\pm} \|_{C^0} < \frac{1}{6C_K} \epsilon .$$
	Then
	\begin{align*}
		& \left| \int f d\mu_{V_i} - \int f d\mu_{V_\infty} \right|	\\
		\le{}& \left| \int f_+ d\mu_{V_i} - \int f_+ d\mu_{V_\infty} \right|	
		+ \left| \int f_- d\mu_{V_i} - \int f_- d\mu_{V_\infty} \right|	\\
		\le{}& \left| \int f_+ - \tilde f_+  d \mu_{V_i} \right|
		+ \left| \int \tilde f_+ d\mu_{V_i} - \int \tilde f_+ d\mu_{V_\infty} \right|
		+ \left| \int f_+ - \tilde f_+ d \mu_{V_\infty} \right|	\\
		&+ \left| \int f_- - \tilde f_-  d \mu_{V_i} \right|
		+ \left| \int \tilde f_- d\mu_{V_i} - \int \tilde f_- d\mu_{V_\infty} \right|
		+ \left| \int f_- - \tilde f_- d \mu_{V_\infty} \right|	\\
		\le{}& 4 \| f - \tilde f \|_{C^0} C_K 
		+ \left| \int \tilde f_+ d\mu_{V_i} - \int \tilde f_+ d\mu_{V_\infty} \right| 	
		+ \left| \int \tilde f_- d\mu_{V_i} - \int \tilde f_- d\mu_{V_\infty} \right| 	\\
		<{}& \frac{2}{3} \epsilon + \left| \int \tilde f_+ d\mu_{V_i} - \int \tilde f_+ d\mu_{V_\infty} \right| 	
		+ \left| \int \tilde f_- d\mu_{V_i} - \int \tilde f_- d\mu_{V_\infty} \right| 	\\
		<{}& \epsilon 		&& ( \forall i \gg 1)
	\end{align*}
	where the last inequality follows by assumption \eqref{super weak convergence}.
	This completes the proof that $d \mu_{V_i} \rightharpoonup d \mu_{V_\infty}.$
	
	Next, we prove the varifold convergence $V_i \rightharpoonup V_\infty$.
	Suppose for the sake of contradiction that $V_i \not \rightharpoonup V_\infty$.
	Then there exists $f \in C^0_c ( G(n, U ) )$, $\epsilon > 0$, 
	and a subsequence $V_{i_j}$ such that 
		$$\left | \int f dV_{i_j} - \int f d V_\infty \right | > \epsilon \qquad \forall j.$$
	Since the $V_{i_j}$ have locally uniformly bounded areas and locally uniformly bounded generalized mean curvatures $H_{i_j} \in L^p_{loc}(U)$,
	Lemma \ref{Lem Compactness} implies there exists an integer rectifiable $n$-varifold $V_\infty'$
	and a subsequence still denoted $V_{i_j}$ such that $V_{i_j} \rightharpoonup V_\infty'$.
	In particular,
		$$d\mu_{V_\infty} = \lim_{j \to \infty} d\mu_{V_{i_j} } = d\mu_{V_\infty'}.$$
	Because $V_{\infty}, V_\infty'$ are integer rectifiable $n$-varifolds with $d\mu_{V_\infty} = d \mu_{V_\infty' }$,
	$V_\infty = V_\infty'$. 
	We then have a contradiction that $V_{i_j} \not \rightharpoonup V_\infty = V_\infty'$ and $V_{i_j} \rightharpoonup V_\infty' = V_\infty.$
	This contradiction proves that in fact $V_i \rightharpoonup V_\infty$.
\end{proof}

\begin{prop}[Monotonicity formula] \label{prop simple mono forla}
	Let $2 \le n < N$, $U \subset \R^N$ open, and $p \in ( n, \infty]$.
	Let $V$ be an integer rectifiable $n$-varifold in $U$ with generalized mean curvature $H \in L^p_{loc}(U)$.
	Let $x_0 \in U$ and $\ol{B_R(x_0)} \subset U$.
	Then for any $0 < \sigma < \rho < R$
	\begin{multline} \label{simple mono forla eqn}
		\left(  e^{ \frac{ \| H \| }{ 1 - \frac{n}{p} } \rho^{1 - \frac{n}{p} } }
		 \frac{ \mu_V ( B_\rho( x_0) ) }{ \rho^n }
		 + e^{ \frac{ \| H \| }{ 1 - \frac{n}{p} } \rho^{1 - \frac{n}{p} } } - 1 \right) \\
		- 
		\left(  e^{ \frac{ \| H \| }{ 1 - \frac{n}{p} } \sigma^{1 - \frac{n}{p} } }
		 \frac{ \mu_V ( B_\sigma( x_0) ) }{ \sigma^n }
		 + e^{ \frac{ \| H \| }{ 1 - \frac{n}{p} } \sigma^{1 - \frac{n}{p} } } - 1 \right) 	\\
		 \ge 
		\int_{B_\rho (x_0) \setminus B_\sigma (x_0) } \frac{ | ( x - x_0)^\perp |^2 }{ |x-x_0|^{n+2} } d \mu_V
		\ge 0 
	\end{multline}
	where $\| H \| = \| H \|_{L^p( B_R(x_0) , d\mu_V ) }$.
	
	In particular,
		$$e^{ \frac{ \| H \| }{ 1 - \frac{n}{p} } \rho^{1 - \frac{n}{p} } }
		 \frac{ \mu_V ( B_\rho( x_0) ) }{ \rho^n }
		 + e^{ \frac{ \| H \| }{ 1 - \frac{n}{p} } \rho^{1 - \frac{n}{p} } } - 1
		 \text{ is non-decreasing in } \rho$$
	and
		$$\theta_V (x_0 ) \doteqdot \lim_{\rho \searrow 0 } \frac{ \mu_V ( B_\rho ( x_0) ) }{ \rho^n} 
		\text{ exists.}$$
\end{prop}

\begin{proof}
	The following proof follows \cite[Ch. 4, \S 4]{SimonGMT}.
	We provide the proof for $p \in (n, \infty)$ and let the reader make the necessary adjustments to the proof in the case of $p = \infty$.
	
	For any $0 < \rho < R/ ( 1+ \epsilon ) $, 
	\begin{equation} \label{simple mono proof 1}
	\frac{ d}{d \rho} ( \rho^{-n} I (\rho ) ) = \rho^{-n}  \frac{d}{d\rho}J(\rho) - \rho^{-n} \int \rho^{-1} ( x - x_0) \cdot H \phi_\epsilon ( | x - x_0| / \rho ) d \mu_V(x)
	\end{equation}
	where $\phi_\epsilon : [0, \infty) \to \R_{\ge 0}$ is a smooth, non-increasing function with $\phi_\epsilon (s) \equiv 1$ for $s \in [0, 1]$ and $\supp \phi_\epsilon \subset [0, 1 + \epsilon]$,
	\begin{gather*}
		x_0 \in U \text{ and } \ol{B_{R  } (x_0)} \subset U,	\\
		I(\rho) \doteqdot \int \phi_\epsilon ( | x - x_0| / \rho ) d \mu_V(x)	\ge 0, \text{ and} \\
		J(\rho) \doteqdot \int \frac{ | ( x- x_0)^\perp |^2 }{ | x - x_0 |^2} \phi_\epsilon ( | x - x_0|/ \rho ) d \mu_V(x) \ge 0 .
	\end{gather*}
	The integral on the right-hand side of \eqref{simple mono proof 1} can be estimated by
	\begin{align*}
		& \left| \rho^{-n} \int \frac{ x - x_0}{\rho } \cdot H \phi_\epsilon \left( \frac{ | x - x_0 |}{\rho} \right) d\mu_V (x) \right|	\\
		\le{}& ( 1 + \epsilon ) \rho^{-n} \| H \|_{L^p (  B_{(1 + \epsilon ) \rho } (x_0 ) , d\mu_V )} ( I (\rho) )^{1 - \frac{1}{p} }	\\
		\le{}&  ( 1 + \epsilon ) \rho^{-n/p} \| H \|_{L^p (  B_{R } (x_0 ) , d\mu_V )} ( \rho^{-n} I (\rho) )^{1 - \frac{1}{p} }	
		&& \left( 0 < \rho < \frac{R}{1 + \epsilon} \right)\\
		\le{}& ( 1 + \epsilon ) \rho^{-n/p} \| H \|_{L^p (  B_{R } (x_0 ) , d\mu_V )} (1 +  \rho^{-n} I (\rho) )	
	\end{align*}
	where in the last step we used $a \ge 0 \implies a^{1 - 1/p} \le 1 + a$.
	Inserting this estimate into \eqref{simple mono proof 1} yields
	\begin{multline}
		\label{simple mono proof 2}
		\frac{ d}{d \rho} ( 1 + \rho^{-n} I (\rho ) ) \ge \rho^{-n}  \frac{d}{d\rho}J(\rho) 
		- ( 1 + \epsilon ) \rho^{-n/p} \| H \|_{L^p (  B_{R } (x_0 ) , d\mu_V )} (1 +  \rho^{-n} I (\rho) )	\\
		  \forall 0 < \rho < R / ( 1 + \epsilon ) .
	\end{multline}
	To simplify the notation, we write $\| H \| = \| H \|_{L^p( B_R(x_0), d \mu_V )}$ for the remainder of the proof.
	Multiplying \eqref{simple mono proof 2} by the integrating factor
		$$F_\epsilon(\rho) \doteqdot e^{ \int_0^\rho ( 1 + \epsilon) \| H \| \tilde \rho^{- n/p}  d \tilde \rho } 
		= e^{\frac{ ( 1 + \epsilon ) \| H \|}{ 1 - n/p} \rho^{1 - n/p}} \ge 1$$
	and using the fact that $\frac{d}{d \rho} J \ge 0$ 
	yields
	$$\frac{ d}{ d \rho} \left( F_\epsilon( \rho ) \rho^{-n} I(\rho) + F_\epsilon( \rho ) - 1 \right) 
	\ge F_\epsilon( \rho ) \rho^{-n} \frac{d}{ d \rho} J
	\ge  \rho^{-n} \frac{d}{ d \rho} J.$$
	Integrating from $\sigma$ to $\rho$ then gives
		$$\left(F_\epsilon ( \rho ) \frac{ I_\epsilon ( \rho ) }{ \rho^n} + F_\epsilon ( \rho ) - 1 \right)
		- \left(F_\epsilon ( \sigma ) \frac{ I_\epsilon ( \sigma ) }{ \sigma^n} + F_\epsilon ( \sigma ) - 1 \right)
		\ge \int_\sigma^\rho \tilde \rho^{-n} \frac{ d}{ d \tilde \rho } J d \tilde \rho .$$
	Using that $\frac{ d}{d \rho} \left(  \phi_\epsilon ( | x - x_0 |/ \rho ) \right)$ is supported on the region $\rho \le | x - x_0 | \le ( 1 + \epsilon ) \rho$,
	it follows that
	\begin{align*}
		\int_\sigma^\rho \tilde \rho^{-n} \frac{ d}{ d \tilde \rho } J d \tilde \rho 
		&= \int_\sigma^\rho \int \frac{  | ( x - x_0)^\perp|^2}{ | x - x_0|^2 } \tilde \rho^{-n} 
		\frac{d}{ d \tilde \rho} \left( \phi_\epsilon ( |x -x_0| / \rho ) \right) \, d\mu_V d \tilde \rho	\\
		&\ge \int_\sigma^\rho \int \frac{  | ( x - x_0)^\perp|^2}{ | x - x_0|^{n+2} } 
		\frac{d}{ d \tilde \rho} \left( \phi_\epsilon ( |x -x_0| / \rho ) \right) \, d\mu_V d \tilde \rho	\\
		&= \int \frac{  | ( x - x_0)^\perp|^2}{ | x - x_0|^{n+2} } 
		\left( \phi_\epsilon ( |x -x_0| / \rho ) - \phi_\epsilon ( |x -x_0| / \sigma )  \right) \, d\mu_V \\
		&\ge \int_{B_\rho(x_0) \setminus B_{(1 + \epsilon) \sigma} ( x_0) }  \frac{  | ( x - x_0)^\perp|^2}{ | x - x_0|^{n+2} }  \, d\mu_V.
	\end{align*}
	Letting $\epsilon \searrow 0$ finally yields
	\begin{multline*}
		\left(  e^{ \frac{ \| H \| }{ 1 - \frac{n}{p} } \rho^{1 - \frac{n}{p} } }
		 \frac{ \mu_V ( B_\rho( x_0) ) }{ \rho^n }
		 + e^{ \frac{ \| H \| }{ 1 - \frac{n}{p} } \rho^{1 - \frac{n}{p} } } - 1 \right) \\
		- 
		\left(  e^{ \frac{ \| H \| }{ 1 - \frac{n}{p} } \sigma^{1 - \frac{n}{p} } }
		 \frac{ \mu_V ( B_\sigma( x_0) ) }{ \sigma^n }
		 + e^{ \frac{ \| H \| }{ 1 - \frac{n}{p} } \sigma^{1 - \frac{n}{p} } } - 1 \right) 	\\
		 \ge 
		\int_{B_\rho (x_0) \setminus B_\sigma (x_0) } \frac{ | ( x - x_0)^\perp |^2 }{ |x-x_0|^{n+2} } d \mu_V
		\ge 0 .
	\end{multline*}
	 
	 In particular, we have monotonicity of 
	 $$ \rho \mapsto e^{ \frac{ \| H \| }{ 1 - \frac{n}{p} } \rho^{1 - \frac{n}{p} } } \frac{ \mu_V ( B_\rho ( x_0) ) }{\rho^{n}} + e^{ \frac{ \| H \| }{ 1 - \frac{n}{p} } \rho^{1 - \frac{n}{p} } } - 1.$$
	 Hence, its limit as $\rho \searrow 0$ exists and thus $\theta_V(x_0)$ is well-defined.
\end{proof}

\bibliographystyle{alpha}
\bibliography{bddHFlows}

\end{document}